\documentclass[12pt]{amsart}
\usepackage{amssymb, latexsym, amsmath, amsfonts, tikz, tikz-cd}

\newtheorem{thm}{Theorem}[section]

\newtheorem{lem}[thm]{Lemma}
\newtheorem{prop}[thm]{Proposition}
\theoremstyle{definition}

\theoremstyle{remark}
\newtheorem{rem}[thm]{Remark}
\numberwithin{equation}{section}
\theoremstyle{remark}
\newtheorem{exam}[thm]{Example}

\newcommand{\mbb}{\mathbb}
\newcommand{\ra}{\rightarrow}
\newcommand{\z}{\zeta}
\newcommand{\pa}{\partial}
\newcommand{\ov}{\overline}
\newcommand{\sm}{\setminus}

\newcommand{\al}{\alpha}
\newcommand{\Om}{\Omega}
\newcommand{\cal}{\mathcal}

\begin{document}
\title{A study on holomorphic isometries of weighted Bergman metrics}
\thanks{The first named author was supported in part by the PMRF Ph.D. fellowship of the Ministry of Education, Government of India.}
\author{Aakanksha Jain and Kaushal Verma}

\address{AJ: Department of Mathematics, Indian Institute of Science, Bangalore 560 012, India }
\email{aakankshaj@iisc.ac.in}

\address{KV: Department of Mathematics, Indian Institute of Science, Bangalore 560 012, India}
\email{kverma@iisc.ac.in}

\begin{abstract} 
For a domain $D \subset \mbb C^n$ and an admissible weight $\mu$ on it, we consider the weighted Bergman kernel $K_{D, \mu}$ and the corresponding weighted Bergman metric on $D$. In particular, motivated by work of Mok, Ng, Chan--Yuan and Chan--Xiao--Yuan among others, we study the nature of holomorphic isometries from the disc $\mbb D \subset \mbb C$ with respect to the weighted Bergman metrics arising from weights of the form $\mu = K_{\mbb D}^{-d}$ for some integer $d \ge 0$. These metrics provide a natural class of examples that give rise to positive conformal constants that have been considered in various recent works on isometries. Specific examples of isometries that are studied in detail include those in which the isometry takes values in $\mbb D^n$ and $\mbb D \times \mbb B^n$ where each factor admits a weighted Bergman metric as above for possibly different non-negative integers $d$. Finally, the case of isometries between polydisks in possibly different dimensions, in which each factor has a different weighted Bergman metric as above, is also presented.
\end{abstract}  

\maketitle 

\section{Introduction}

\noindent Let $D \subset \mbb C^n$ be a domain and $\mu$ a positive measurable function on it. Let $L^2_{\mu}(D)$ denote the space of all functions on $D$ that are square integrable with respect to $\mu dV$, where $dV$ denotes standard Lebesgue measure, and set $\mathcal O_{\mu}(D) = L^2_{\mu}(D) \cap \mathcal O(D)$. The class of weights $\mu$ for which $\mathcal O_{\mu}(D) \subset L^2_{\mu}(D)$ is closed, and for any $z \in D$, the point evaluations $z \mapsto f(z) \in \mbb C$ are bounded on $\mathcal O_{\mu}(D)$ were considered by Pasternak--Winiarski in \cite{PW2}, \cite{PW1} and termed {\it admissible} therein. In this situation, there is a reproducing kernel $K_{D, \mu}(z, w)$ which is the weighted Bergman kernel with weight $\mu$. Let $K_{D, \mu}(z) = K_{D, \mu}(z, z)$ for $z \in D$. A sufficient condition for $\mu$ to be admissible is that $\mu^{-a}$ be locally integrable in $D$ with respect to $dV$ for some $a >0$ (see \cite{PW1}). In particular, this holds if $1/\mu \in L^{\infty}_{\rm loc}(D)$. 

\medskip

Within our framework, we shall consider weights that extend continuously to $\overline{D}$ and are not necessarily positive on $\partial D$. This condition is not a restrictive one and will be seen to arise naturally in what follows. 

\medskip

As usual, the classical Bergman kernel and its restriction to the diagonal will be denoted by $K_D(z, w)$ and $K_D(z)$ respectively, and this corresponds to the case $\mu \equiv 1$.

\medskip

The purpose of this note is to study holomorphic maps from the disc $\mathbb D \subset \mathbb C$ that are isometries with respect to metrics arising from weighted Bergman kernels. Two classes of weights will be considered here - one, those that are non-vanishing on the boundary and two, a specific class of examples of weights that do vanish on the boundary.

\medskip

Let $D \subset \mbb C^n$ be a bounded domain and $\mu$ a positive continuous weight on $\ov D$. Then $L^2(D) = L^2_{\mu}(D)$ and therefore, $K_{D, \mu}$, which can be written as
\[
K_{D,\mu}(z) = \sup\{\vert f(z) \vert^2 : f \in \cal O_{\mu}(D), \Vert f \Vert_{\mu} = 1 \}
\]
is a positive plurisubharmonic function on $D$. The weighted Bergman metric 
\[
ds^2_{D, \mu} = \sum_{i, j} \frac{\pa^2}{\pa z_i \pa \ov z_j} \log K_{D, \mu}(z) \; dz_i \otimes d \ov z_j
\]
is, therefore, well-defined on $D$. When $\mu \equiv 1$, $ds^2_{D}$ will denote the classical (= unweighted) Bergman metric on $D$.

\begin{thm}\label{1.3}
Let $D_1 \subset \mbb C^n, D_2 \subset \mbb C^m$ be bounded domains and suppose that $\mu_1, \mu_2$ are positive continuous weights on $\ov D_1, \ov D_2$ respectively. Let $f : (D_1, ds^2_{D_1, \mu_1}) \ra (D_2, ds^2_{D_2, \mu_2})$
be a holomorphic isometry. Then, there exist functions $s_0 \in \mathcal O_{\mu_1}(D_1)$ and $r_0 \in \mathcal O_{\mu_2}(D_2)$ such that
\[
s_0(z) \; \ov{s_0(w)} \; K_{D_2, \mu_2}(f(z), f(w)) = r_0(f(z)) \; \ov{r_0(f(w))} \; K_{D_1, \mu_1}(z, w)
\]
for all $z, w \in D_1$. In particular, if $D_1 = D_2 = \mbb D$ and $\mu_1 \equiv 1$, then $f$ is an automorphism of $\mbb D$.

\end{thm}

Some comments are in order. Theorem \ref{1.3} and its proof are directly motivated by Mok's work (see \cite{Mok1}) on the extension of germs of local isometries of the Bergman metric. The functional relation between the weighted Bergman kernels does not require $f$ to be defined on all of $D_1$. In fact, fix base points $a, b$ in $D_1, D_2$ respectively and suppose that $f$ is a germ of a holomorphic isometry defined near $a$ such that $f(a) = b$. The proof shows that the same functional equation holds for $z, w$ near $a$. However, given such a germ of a holomorphic isometry, we do not know whether this germ extends to all of $D_1$ in the spirit of \cite{Mok1}. 

\medskip

For a  biholomorphism $f: D_1 \ra D_2$ and an admissible weight $\mu$ on $D_1$, the weight $\mu\circ f^{-1}$ on $D_2$ is also admissible. Therefore, $ K_{D_2, \mu \circ f^{-1}}$ is well-defined. The weighted kernels transform as
\begin{equation}\label{weighted_trans}
K_{D_1, \mu}(z, w) = Jf(z) K_{D_2, \mu \circ f^{-1}} (f(z), f(w)) \ov{Jf(w)}.
\end{equation}
Here, $Jf$ is the Jacobian determinant of $f$. Hence, $f$ is an isometry with respect to the metrics corresponding to these weighted kernels. Theorem \ref{1.3} talks about isometries, where the weights need not vary with the holomorphic function.

\medskip

Moving ahead, we now consider a smoothly bounded planar domain $D \subset \mbb C$ and weights of the form $\mu_d(z) = K_D^{-d}(z)$ for some integer $d \ge 0$. The Bergman kernel on the diagonal for such domains blows up at the boundary and hence the weights $\mu_d$ can be regarded to extend continuously to the boundary $\pa D$ by setting $\mu_d \equiv 0$ on $\pa D$. Further, they are admissible since $1/\mu_d = K_D^d$ is locally bounded on $D$. The corresponding weighted Bergman kernel will be denoted by $K_{D, d}$. The intrinsic advantage of working with $\mu_d$ as a weight is that $K_{D, d}$ transforms much like the unweighted kernel under biholomorphisms. It can be checked that 
\begin{equation}\label{Trans formula}
K_{D_1, d}(z, w) = (Jf(z))^{d+1} K_{D_2, d}(f(z), f(w)) (\ov{Jf(w)})^{d+1}
\end{equation}
for a biholomorphism $f : D_1 \ra D_2$ -- this holds for domains in $\mbb C^n$ also. Since $\mu_d = 0$ on $\pa D$, $L^2(D) \subset L^2_{\mu_d}(D)$ which implies that $K_{D, d}(z) > 0$ everywhere and hence the corresponding metric $ds^2_{D, d}$ is well-defined on $D$. The transformation rule (\ref{Trans formula}) shows that biholomorphisms are isometries of the weighted metrics $ds^2_{D, d}$. As before, one could ask for an understanding of the isometries with respect to these weighted metrics. We, therefore, turn to the question of studying isometries of the metrics $ds^2_{D, d}$. We will focus on the case when $D = \mbb D$ and write $ds^2_d$, instead of $ds^2_{\mbb D, d}$, for brevity. The conclusions are computationally complete here and moreover, they can be regarded as weighted analogs of the work of S-C Ng \cite{Ng} and others. The point here is $ds^2_{\mbb D, d}$ provides a natural class of metrics that give rise to positive conformal constants that have been considered in various other works on isometries.

\begin{thm}\label{1.5}
Let $d, d_1, d_2, \ldots, d_n$ be non-negative integers and suppose that
\[
f : (\mbb D, ds^2_d) \ra (\mbb D, ds^2_{d_1}) \times (\mbb D, ds^2_{d_2}) \times \cdots \times (\mbb D, ds^2_{d_n})
\]
is a holomorphic isometry such that all the component functions of $f$ are non-constant. Then,
\[
d+1 \le \sum_{i = 1}^n (d_i+1).
\]
In case equality holds, there are unimodular constants $\al_1, \al_2, \ldots, \al_n$ such that
\[
f(z) = (\al_1 z, \al_2 z, \ldots, \al_n z)
\]
after a possible reparametrization. Moreover, for every $1 \le i \le n$, there exist integers $1 \le i = k_1, k_2, \ldots, k_{j_i} \le n$ such that
\[
\sum_{r=1}^{j_i} (d_{k_r} + 1) = d+1.
\]
In particular, $d_i \le d$ for each $1 \le i \le n$.
\end{thm}

A complete classification of holomorphic isometries from the disc into the bidisc can be found in \cite{Ng} and its generalization to the case with positive conformal constants in \cite{CXY}. The case of holomorphic isometries into the tridisc has been explored in \cite{CY, Ng}. The next theorem clarifies the situation of isometries into the polydisc in $\mbb C^3$ by giving the complete characterization.

\begin{thm}\label{1.6}
Let $n = 3$ in Theorem \ref{1.5}, that is, let
\[
f : (\mbb D, ds^2_d) \ra (\mbb D, ds^2_{d_1}) \times (\mbb D, ds^2_{d_2}) \times (\mbb D, ds^2_{d_3})
\]
be a holomorphic isometry such that all the component functions of $f$ are non-constant. A complete characterization of all possible isometries $f$ is as follows:  
\begin{enumerate}
\item[(a)] $d_1=d_2=d_3=d$. In this case, up to a reparametrization, $f$ is either the cube root embedding or $f$ can be factorized as 
\[
f(z)=\left(\alpha_1(z),\alpha_2(\beta_1(z)),\beta_2(\beta_1(z))\right)
\]
where $z\mapsto (\alpha_i(z),\beta_i(z))$ is a holomorphic isometry from $(\mathbb{D},ds^2_{\mathbb{D}})$ to $(\mathbb{D},ds^2_{\mathbb{D}})^2$ for both $i=1,2$.
\item[(b)] $\sum\limits_{i=1}^{3}(d_i+1)=d+1$. In this case, up to a reparametrization,
\[
f(z)=(\alpha_1 z, \alpha_2 z,\alpha_3 z)
\]
where $\alpha_i\in\mathbb{C}$ are unimodular constants.
\item[(c)] $d_{k_1}=d$ and $\sum\limits_{i=2,3}(d_{k_i}+1)=d+1$, where $k_i\neq k_j$, $i\neq j$ and $k_i\in\{1,2,3\}$. In this case, up to a reparametrization,
\[
f(z)=(c\,\alpha(z),\alpha(z),\beta(z))
\]
where $c\in\mathbb{C}$ is a unimodular constant and $z\mapsto(\alpha(z),\beta(z))$ is a holomorphic isometry from $(\mathbb{D},ds^2_{\mathbb{D}})$ to $(\mathbb{D},ds^2_{\mbb D})^2$.
\item[(d)] $\sum\limits_{i=1,3}(d_{k_i}+1)=d+1=\sum\limits_{i=2,3}(d_{k_i}+1)$, where $k_i\neq k_j$, $i\neq j$ and $k_i\in\{1,2,3\}$. In this case, up to a reparametrization,
\[
f(z)=(z,\alpha(z),\beta(z))
\]
where $z\mapsto (\alpha(z),\beta(z))$ is a holomorphic isometry from $(\mathbb{D},ds^2_{\mbb D})$ to $(\mathbb{D},ds^2_{\mbb D})^2$.
\end{enumerate}
\end{thm}

The concluding statements below are motivated by Chan--Yuan \cite{CY} and Ng \cite{Ng} respectively.

\begin{thm}\label{1.7}
Let $d, d_1, d_2$ be non-negative integers and
\[
f : (\mbb D, ds^2_d) \ra (\mbb D, ds^2_{d_1}) \times (\mbb B^n, ds^2_{d_2})
\]
be a holomorphic isometry. Then,
\[
2(d+1) \le 2(d_1 + 1) + (n+1)(d_2 + 1)
\]
and equality holds precisely when, up to reparametrization, $f(z) = (cz, c_1z, \ldots, c_nz)$ where $\vert c \vert = 1$ and the vector $(c_1, c_2, \ldots, c_n) \in \pa \mbb B^n$.
\end{thm}

\begin{thm}\label{1.8}
Let $c_1, c_2, \ldots, c_m$ and $d_1, d_2, \ldots, d_n$ be non-negative integers and
\[
f : (\mbb D, ds^2_{c_1}) \times \cdots \times (\mbb D, ds^2_{c_m}) \ra (\mbb D, ds^2_{d_1}) \times \cdots \times (\mbb D, ds^2_{d_n})
\]
be a holomorphic isometry. Let $f = (f_1, f_2,\ldots, f_n)$. Then, after a possible reparametrization, there exist positive integers $n_1, n_2, \ldots, n_m$ with $n_1 + n_2 + \ldots + n_m = n$ such that $f_1, \ldots, f_{n_1}$ depend on $z_1$ only, $f_{n_1+1}, \ldots, f_{n_1+n_2}$ depend on $z_2$ only, and so on till the final subset of components $f_{n_1 + n_2 + \ldots + n_{m-1} +1}, \ldots, f_n$ depend on $z_m$ only.

\medskip

In particular, the $i$-th group of the above $m$ groups of component functions constitutes a holomorphic isometry from $(\mbb D, ds^2_{c_i})$ into the product $(\mbb D, ds^2_{d_{n_1 + \ldots + n_{i-1} + 1}}) \times \cdots \times (\mbb D, ds^2_{d_{n_1 + \ldots + n_i}})$; here $n_0 = 0$.

\medskip

Finally, 
\[
c_i + 1 \le \sum_{k = n_1 + \ldots + n_{i-1} +1}^{n_1 + \ldots + n_i} (d_k + 1)
\]
and equality holds precisely when $f_k = a_kz_i$ for all $n_1 + \ldots + n_{i-1} +1 \le k \le n_1 + \ldots + n_i$, where all the $a_k$'s are unimodular constants.
\end{thm}

\medskip

\noindent {\it Acknowledgements}: The authors would like to thank the referee for carefully reading the article and providing helpful suggestions.

\section{Weights that do not vanish on the boundary}

\noindent In this section, we will prove Theorem \ref{1.3} using ideas from \cite{Mok1}.

\begin{proof}[Proof of Theorem \ref{1.3}]

For a bounded domain $D \subset \mathbb{C}^n$ and a positive continuous weight $\mu$ on $\overline{D}$, denote the length of a tangent vector $X=(X_1,\ldots,X_n)$ at a point $z\in\Omega$ with respect to the metric $ds^2_{D,\mu}$ by $\Vert X\Vert_{z,D,\mu}$. Then,
\[
\Vert X\Vert_{z,D,\mu}^2=\sum_{i,j} \frac{\partial^2}{\partial z_i\partial\ov{z}_j}\log K_{D,\mu}(z) X_i\ov{X}_j.
\]
Since $f$ is an isometry, 
\[
\Vert X\Vert_{z,D_1,\mu_1} =\Vert df(z) X\Vert_{f(z),D_2,\mu_2}
\]
for $z \in D_1$ and $X \in \mbb C^n$, and this implies that
 \begin{equation*}
\pa \ov \pa \log K_{D_2, \mu_2}(f(z)) = \pa \ov \pa \log K_{D_1, \mu_1}(z)
\end{equation*}
for all $z\in D_1$. Thus,
$\log K_{D_2,\mu_2}(f(z)) -\log K_{D_1,\mu_1}(z)$ is pluriharmonic and hence can be locally written as the real part of a holomorphic function. 

\begin{proof}[Case 1.] Assume that $0\in D_1,D_2$ and $f(0)=0$.
There exists a holomorphic function $\psi$ near the origin such that
\begin{equation}\label{Iso 1}
\log K_{D_2,\mu_2}(f(z))=\log K_{D_1,\mu_1}(z)+ \Re \psi(z),
\end{equation}
for $z$ near $0$.
The hypotheses on the weights imply that for $i = 1, 2$, $L^2_{\mu_i}(D_i) = L^2(D_i)$ and that $\cal O_{\mu_i}(D_i)$ contain the constants. By considering the bounded functional $g \mapsto g - g(0)$ on $\cal O_{\mu_i}(D_i)$, it is possible to choose orthonormal bases $\{s_i\}_{i=0}^{\infty}$ and $\{r_i\}_{i=0}^{\infty}$ of $\mathcal{O}_{\mu_1}(D_1)$ and $\mathcal{O}_{\mu_2}(D_2)$ respectively so that $s_i(0)=r_i(0)=0$ for all $ i\geq 1$, and $s_0(0)\neq 0\neq r_0(0)$. For $z$ near $0$,
\begin{equation*}
K_{D_1,\mu_1}(z) =
\sum_{i=0}^{\infty}\left\lvert s_i(z)\right\rvert^2
=
\left\lvert s_0(z)\right\rvert^2 + \sum_{i=1}^{\infty}\left\lvert s_i(z)\right\rvert^2 = \left\lvert s_0(z)\right\rvert^2 \left( 1+\sum_{i=1}^{\infty}\left\lvert \frac{s_i(z)}{s_0(z)}\right\rvert^2 \right),
\end{equation*}
and by repeating these steps for $K_{D_2,\mu_2}$, 
\[
K_{D_1,\mu_1}(z)=\left\lvert s_0(z)\right\rvert^2 \tilde{K}_{D_1,\mu_1}(z)
\quad
\text{and}
\quad
K_{D_2,\mu_2}(z)=\left\lvert r_0(z)\right\rvert^2 \tilde{K}_{D_2,\mu_2}(z),
\]
where
\[
\tilde{K}_{D_1,\mu_1}(z)=1+\sum_{i=1}^{\infty}\left\lvert \frac{s_i(z)}{s_0(z)}\right\rvert^2
\quad
\text{and}
\quad
\tilde{K}_{D_2,\mu_2}(z)=1+\sum_{i=1}^{\infty}\left\lvert \frac{r_i(z)}{r_0(z)}\right\rvert^2
\]
for $z$ near the origin.
Note that $\sum_{i=1}^{\infty}\left\lvert s_i/s_0\right\rvert^2<1$ in a neighborhood of $0$. By working with the Taylor expansion of $\log(1+x)$, $\log \tilde{K}_{D_1,\mu_1}(z)$ can be expressed as a convergent sum of a countable number of functions of the form $\pm \left\lvert \theta_k\right\rvert^2$, where each $\theta_k$ is a holomorphic function vanishing at $0$. Therefore, by expanding the $\theta_k$'s further as a convergent power series, it follows that the Taylor expansion of $\log \tilde{K}_{D_1,\mu_1}(z)$ at $0$ does not contain terms of pure type $z^I$ and $\ov{z}^I$. Here $z^I = z_1^{i_1}\ldots z_n^{i_n}$ for $z=(z_1,\ldots,z_n)$ and $I=(i_1,\ldots,i_n)$. Similarly, since $f(0) = 0$, the same holds for the Taylor expansion of $\log \tilde{K}_{D_2,\mu_2}(f(z))$ at the origin.

\medskip

Now $(\ref{Iso 1})$ can be rewritten as
\begin{equation*}
\log \left\lvert r_0(f(z))\right\rvert^2 + \log \tilde{K}_{D_2,\mu_2}(f(z))=\log \left\lvert s_0(z)\right\rvert^2 + \log \tilde{K}_{D_1,\mu_1}(z)+ \Re \psi(z).
\end{equation*}
Since $f(0) = 0$ and neither of $r_0, s_0$ vanish at the origin, both
$\log\vert r_0(f(z))\vert^2$ and $\log\vert s_0(z)\vert^2$ are pluriharmonic near the origin in $D_1$. Thus,
\[
\partial\ov{\partial} \log \tilde{K}_{D_2,\mu_2}(f(z))
=\partial\ov{\partial} \log \tilde{K}_{D_1,\mu_1}(z)
\]
and hence, there exists a holomorphic function $\tilde{\psi}$ near the origin such that
\begin{equation}\label{Iso 2}
    \log \tilde{K}_{D_2,\mu_2}(f(z))=\log \tilde{K}_{D_1,\mu_1}(z)+ \Re \tilde{\psi}(z).
\end{equation}
Since the Taylor expansion of $\Re(\tilde{\psi})=\tilde{\psi}+\overline{\tilde{\psi}}$ at the origin consists precisely of terms of pure type, it follows by comparing the two sides of $(\ref{Iso 2})$ that $\Re \tilde{\psi}$ is a constant. Hence,
\[
\tilde{K}_{D_2,\mu_2}(f(z))=a\,\tilde{K}_{D_1,\mu_1}(z)
\]
for some constant $a>0$. This leads to
\begin{equation*}
K_{D_2,\mu_2}(f(z)) = \left\lvert r_0(f(z))\right\rvert^2 \tilde{K}_{D_2,\mu_2}(f(z)) =
a\,\left\lvert r_0(f(z))\right\rvert^2 \tilde{K}_{D_1,\mu_1}(z) =
a\,\frac{\left\lvert r_0(f(z))\right\rvert^2}{\left\lvert s_0(z)\right\rvert^2}  K_{D_1,\mu_1}(z).
\end{equation*}
Since $K_{D_1,\mu_1}(0)=\vert s_0(0)\vert^2\,\,\text{and}\,\,K_{D_2,\mu_2}(0)=\vert r_0(0)\vert^2$, we observe that
\[
a=\frac{K_{D_2,\mu_2}(0)\,\left\lvert s_0(0)\right\rvert^2}{K_{D_1,\mu_1}(0)\,\left\lvert r_0(0)\right\rvert^2}=1
\]
by putting $z = 0$.
Thus, for $z$ near $0$,
\begin{equation}\label{Iso 3}
\left\lvert s_0(z)\right\rvert^2 \,K_{D_2,\mu_2}(f(z))
=
\left\lvert r_0(f(z))\right\rvert^2 \,K_{D_1,\mu_1}(z).
\end{equation}
Now polarize this equality. For $z, w$ near the origin, write
\[
s_0(z)\,\overline{s_0(w)} \,K_{D_2,\mu_2}(f(z),f(w))
=
r_0(f(z))\,\overline{r_0(f(w))} \, K_{D_1,\mu_1}(z,w) + H(z,w),
\]
where $H(z,w)=\sum_{(I,J)} h_{IJ}z^I\bar{w}^J$ is holomorphic in $z$ and anti-holomorphic in $w$. On restriction to the diagonal $\{z=w\}$, we have $H(z,z)=0$, i.e., $\sum h_{IJ}z^I\bar{z}^J=0$. So, $h_{IJ}=0$ for all $I,J$, and hence $H\equiv 0$. Thus,
\begin{equation}\label{Iso 4}
    s_0(z)\,\overline{s_0(w)}\,K_{D_2,\mu_2}(f(z),f(w))= r_0(f(z))\,\overline{r_0(f(w))}\, K_{D_1,\mu_1}(z,w).
\end{equation}
for $z, w$ near the origin and therefore for all $z, w \in D_1$.
\end{proof}

\begin{proof}[Case 2.] Otherwise, choose $p\in D_1$ and let $q = f(p) \in D_2$. The translations $\Lambda_1:z\mapsto z-p$ and $\Lambda_2:z\mapsto z-q$ move the base points $p, q$ to the origin that are contained in the bounded images $\Hat{D}_i=\Lambda_i(D_i)$ for $i=1,2$. The weights $\Hat{\mu}_i=\mu_i\circ\Lambda^{-1}_i$ are positive and continuous on the closure of $\Hat{D}_i$ for $i=1,2$. 

It follows from $(\ref{weighted_trans})$ that $\Lambda_1^{-1}:(\Hat{D}_1,ds^2_{\Hat{D}_1,\Hat{\mu}_1})\rightarrow(D_1,ds^2_{D_1,\mu_1})$ and $\Lambda_2:(D_2,ds^2_{D_2,\mu_2})\rightarrow(\Hat{D}_2,ds^2_{\Hat{D}_2,\Hat{\mu}_2})$ are isometries and therefore, $g=\Lambda_2\circ f\circ \Lambda^{-1}_1:(\Hat{D}_1,ds^2_{\Hat{D}_1,\Hat{\mu}_1})\rightarrow(\Hat{D}_2,ds^2_{\Hat{D}_2,\Hat{\mu}_2})$ is a holomorphic isometry such that $g(0)=0$. As before, choose orthonormal bases $\{s_i\}_{i=0}^{\infty}$ and $\{r_i\}_{i=0}^{\infty}$ of $\mathcal{O}_{\mu_1}(D_1)$ and $\mathcal{O}_{\mu_2}(D_2)$ respectively so that 
$s_i(p)=0$, $r_i(q)=0$ for $ i\geq 1$, and $s_0(p)\neq 0$, $r_0(q)\neq 0$. Observe that the pull-backs $\Hat{s}_i=s_i\circ\Lambda_1^{-1}$ and $\Hat{r}_i=r_i\circ\Lambda_2^{-1}$, $i \ge 0$, are orthogonal, i.e., $\langle \Hat{s}_i,\Hat{s}_j\rangle_{\mathcal{O}_{\Hat{\mu}_1}(\Hat{D}_1)} = \langle \Hat{r}_i,\Hat{r}_j\rangle_{\mathcal{O}_{\Hat{\mu}_2}(\Hat{D}_2)}=\delta_{ij}$ for all $i, j \ge 0$. Therefore,
$\{\Hat{s}_i\}_{i=0}^{\infty}$ and $\{\Hat{r}_i\}_{i=0}^{\infty}$ form orthonormal bases of $\mathcal{O}_{\Hat{\mu}_1}(\Hat{D}_1)$ and $\mathcal{O}_{\Hat{\mu}_2}(\Hat{D}_2)$ respectively, such that $\Hat{s}_i(0)=\Hat{r}_i(0)=0$ for $i\geq 1$ and $\Hat{s}_0(0)\neq 0\neq\Hat{r}_0(0)$.

\medskip 

It now follows from Case 1 that 
\[
\Hat{s}_0(z)\,\overline{\Hat{s}_0(w)}\,
K_{\Hat{D}_2,\Hat{\mu}_2}(g(z),g(w))=
\Hat{r}_0(g(z))\,\overline{\Hat{r}_0(g(w))}\, K_{\Hat{D}_1,\Hat{\mu}_1}(z,w)
\]
for $z,w\in \Hat{D}_1$. 
Finally, using the transformation formula $K_{D_i,\mu_i}(z,w)=K_{\Hat{D}_i,\Hat{\mu}_i}(\Lambda_i(z),\Lambda_i(w))$ for $i=1,2$, we obtain 
\begin{eqnarray*}
s_0(z)\,\overline{s_0(w)}\,K_{D_2,\mu_2}(f(z),f(w))
&=&
s_0(z)\,\overline{s_0(w)}\,K_{\Hat{D}_2,\Hat{\mu}_2}(\Lambda_2( f(z)),\Lambda_2(f(w)))\\
&=&
\Hat{s}_0(\Lambda_1(z))\,\overline{\Hat{s}_0(\Lambda_1(w))}\,K_{\Hat{D}_2,\Hat{\mu}_2}(g(\Lambda_1(z)),g(\Lambda_1(w)))\\
&=&
\Hat{r}_0(g(\Lambda_1(z)))\,\overline{\Hat{r}_0(g(\Lambda_1(w)))}\, 
K_{\Hat{D}_1,\Hat{\mu}_1}(\Lambda_1(z),\Lambda_1(w))\\
&=&
r_0(f(z))\,\overline{r_0(f(w))}\, K_{D_1,\mu_1}(z,w).
\end{eqnarray*}
for every $z,w\in D_1$.
\end{proof}

Finally, if $D_1 = D_2 = \mbb D$ and $\mu_1 \equiv 1$ we can choose $p=0$. In this case, $s_0\equiv 1/\sqrt{\pi}$ and $r_0$ is chosen so that $r_0(f(0)) \not= 0$. Therefore,
\[
K_{D_2,\mu_2}(f(z)) = \pi\,\left\lvert r_0(f(z))\right\rvert^2\,
K_{\mathbb{D}}(z),\quad z\in\mathbb{D}.
\]
Let $Z = \{r_0 = 0\}$ which, if non-empty, is discrete in $\mbb D$. Let $z_n \ra \pa \mbb D$. Then $f(z_n)$ cannot cluster at a point in $\mbb D \sm Z$, for if it did, then $r_0(f(z_n))$ would be non-vanishing and $K_{\mbb D}(z_n) \ra \infty$ and this would imply that $K_{D_2, \mu_2}(f(z_n)) \ra \infty$ -- this is a contradiction since $K_{D_2, \mu_2}(f(z_n))$ must remain bounded as $f(z_n)$ is assumed to cluster at a point in $\mbb D \sm Z \subset \mbb D$. Thus, $f : \mbb D \ra \mbb D \sm Z$ is proper. Now, by the continuity of $f$, the cluster set of $\pa \mbb D$ under $f$ is connected. This means that the cluster set cannot intersect $Z$, for if it did, then the cluster set would reduce to a point (in $Z$) and hence $f$ would be a constant. Contradiction. Therefore, the cluster set of $\pa \mbb D$ under $f$ is contained entirely in $\pa \mbb D$ which implies that $f$ is a proper self-map of $\mbb D$.
Since $f$ is an isometry, $f'$ cannot vanish. Therefore, $f$ is a covering and hence an automorphism of $\mbb D$. That is, $f=e^{i\theta} \varphi_{\zeta}$ where
\[
\varphi_{\zeta}(z)=\frac{z-\zeta}{1-\bar{\zeta}z},\quad z\in\mathbb{D}.
\]
for some $\zeta\in \mathbb{D}$ and $\theta\in\mathbb{R}$. Therefore,
\begin{eqnarray*}
K_{\mathbb{D},\mu_2}(z,w)
&=&
\pi\,
r_0(z) \, \ov{r_0(w)}\,
K_{\mathbb{D}}(f^{-1}(z),f^{-1}(w))\\ 
&=& 
\pi\,
r_0(z) \, \ov{r_0(w)}\, K_{\mathbb{D}}(e^{-i\theta}\varphi_{-\zeta}(z),e^{-i\theta}\varphi_{-\zeta}(w))\\
&=&
\frac{r_0(z)\,\overline{r_0(w)}}{\left(1-\varphi_{-\zeta}(z)\,\overline{\varphi_{-\zeta}(w)}\right)^2}
\end{eqnarray*}
for $z,w\in \mathbb{D}$.
\end{proof}

\begin{rem}
First, the assumption that $\mu_1, \mu_2$ be positive continuous weights is not really necessary as the proof shows. What is essential is that the spaces $\mathcal O_{\mu_i}(D_i)$ contain the constants and be large enough for the metrics $ds^2_{D_i, \mu_i}$ to be defined.

Second, in the last part of the proof, assume that $D_1 = \mbb D$, $\mu_1 \equiv 1$ and that $D_2 \subset \mbb C^n$ is a bounded domain. The given reasoning shows that the isometry $f$ must be a proper holomorphic map from $\mbb D$ onto $D_2 \sm Z$, where $Z = \{r_0 = 0\} \subset D_2$. In case, $D_2 \subset \mbb C$, then $f: \mbb D \ra D_2$ must be an unbranched proper mapping.
\end{rem}

\section{Proof of Theorem \ref{1.5}}

We will first compute $K_{\mathbb{D},d}$ for $d\geq 0$. Note that $\mu_d(z)=\pi^d (1-\left\vert z\right\rvert^2)^{2d}$. For non-negative integers $m\neq k$,
\begin{equation*}
\langle z^m,z^k\rangle_{L^2_{\mu_d}(\mathbb{D})} =
\pi^d \int_{\mathbb{D}}z^m \bar{z}^k (1-\left\vert z\right\rvert^2)^{2d}\,dV(z) = 0
\end{equation*}
and for a non-negative integer $m$,
\begin{eqnarray*}
\langle z^m,z^m\rangle_{L^2_{\mu_d}(\mathbb{D})}
&=&
2 \, \pi^{d+1}\int_0^1 r^{2m+1} (1-r^2)^{2d}\,dr
=
\pi^{d+1}\int_0^1 s^m (1-s)^{2d}\,ds
\\
&=&
\pi^{d+1}\frac{\Gamma(m+1)\Gamma(2d+1)}{\Gamma(m+2d+2)}
=
\pi^{d+1} \frac{m! \, (2d)!}{(m+2d+1)!}.
\end{eqnarray*}
Hence, $\left\lbrace \sqrt{\frac{1}{\pi^{d+1}}\frac{(m+2d+1)!}{m! \, (2d)!}}z^m\right\rbrace_{m=0}^{\infty}$ is an orthonormal basis of the space $\mathcal{O}_{\mu_d}(\mathbb{D})$. Therefore,
\begin{equation*}
K_{\mathbb{D},d}(z,w)
=\sum_{m=0}^{\infty}\frac{1}{\pi^{d+1}}\frac{(m+2d+1)!}{m! \, (2d)!}z^m\ov{w}^m =
\frac{(2d+1)}{\pi^{d+1}}\frac{1}{(1-z\ov{w})^{2d+2}}
\end{equation*}
for $z,w\in\mathbb{D}$.

\medskip

As a warm-up, we begin with an observation. 

\begin{prop}\label{1.4}
For non-negative integers $d_1 \not= d_2$, there cannot exist a holomorphic isometry
\[
f : (\mbb D, ds^2_{d_1}) \ra (\mbb D, ds^2_{d_2}).
\]
\end{prop}

\begin{proof}
Let $f: (\mathbb{D},ds^2_{d_1})\rightarrow (\mathbb{D},ds^2_{d_2})$ be a holomorphic isometry. Here, $d_1,d_2$ are non-negative integers. If $f(0)=a$, then choose an automorphism $\phi_a$ of $\mathbb{D}$ such that $\phi_a(a)=0$. Since the composition of two isometries remains an isometry, $\phi_a\circ f:(\mathbb{D}, ds^2_{d_1})\rightarrow (\mathbb{D}, ds^2_{d_2})$ is a holomorphic isometry that fixes the origin. Therefore, without loss of generality, we can assume that $f(0)=0$. Since $f$ is an isometry, 
\[
\frac{\partial^2}{\partial z\partial\bar{z}} \log K_{\mbb D, d_2}(f(z))
=\frac{\partial^2}{\partial z\partial\bar{z}} \log K_{\mbb D, d_1}(z),\quad z\in \mathbb{D}
\]
which simplifies to
\[
\frac{2d_2+2}{\left(1-\left\lvert f(z)\right\rvert^2\right)^2}\,\left\lvert f'(z)\right\rvert^2=\frac{2d_1+2}{\left(1-\left\lvert z\right\rvert^2\right)^2}.
\]
By substituting $f(z)=\sum_{n=1}^{\infty}\frac{f^{n}(0)}{n!}z^n$, we obtain
\begin{multline*}
\frac{2d_2+2}{2d_1+2}\,(1+\left\lvert z\right\rvert^4-2\left\lvert z\right\rvert^2)\,\left(\sum_{n,m=0}^{\infty}\frac{f^{n+1}(0)}{n!}\frac{\overline{f^{m+1}(0)}}{m!}z^n\bar{z}^m\right)
\\
=1+\left(\sum_{n,m=1}^{\infty}\frac{f^n(0)}{n!}\frac{\overline{f^m(0)}}{m!}z^n\bar{z}^m\right)^2
-2\left(\sum_{n,m=1}^{\infty}\frac{f^n(0)}{n!}\frac{\overline{f^m(0)}}{m!}z^n\bar{z}^m\right).
\end{multline*}
On comparing the constant term, we get
\[
(2d_2+2)\left\lvert f'(0)\right\rvert^2=2d_1+2.
\]
Thus, $f'(0)\neq 0$. On comparing the coefficients of $z$, it follows that $f''(0)=0$ and therefore,
\[
-2(2d_2+2)\left\lvert f'(0)\right\rvert^2=-2(2d_1+2)\left\lvert f'(0)\right\rvert^2
\]
on comparing the coefficients of $\left\lvert z\right\rvert^2$. This implies that $d_1=d_2$.
\end{proof}

\begin{proof}[Proof of Theorem \ref{1.5}]
For brevity, we will write $K_d$, $K_{d_i}$, instead of $K_{\mathbb{D},d}$, $K_{\mathbb{D},d_i}$. By composing each component of $f=(f_1,f_2,\ldots,f_n)$ with a suitable automorphism of $\mbb D$, we may assume that $f(0) = 0$ while retaining the fact that the reparametrized $f$ is an isometry. Therefore, 
\[
\sum_{i=1}^n \left(\frac{\partial^2}{\partial z\partial\bar{z}}\log K_{d_i}\right)(f_i(z))\, \vert f_i'(z)\vert^2 \,dz \otimes d\ov{z}
=
\frac{\partial^2}{\partial z\partial\bar{z}}\log K_{d}(z) \,dz \otimes d\ov{z}
\]
which is equivalent to saying that
\begin{equation}\label{Ng1}
    \sum_{i=1}^n \frac{\partial^2}{\partial z\partial\bar{z}}\log K_{d_i}(f_i(z))=\frac{\partial^2}{\partial z\partial\bar{z}}\log K_{d}(z), \quad z\in \mathbb{D}.
\end{equation}
Therefore, there exists a holomorphic function $\psi$ on $\mathbb{D}$ such that
\begin{equation}\label{Ng2}
\sum_{i=1}^n\log K_{d_i}(f_i(z))=\log K_{d}(z)+ \Re \psi(z),\quad z\in\mathbb{D}.
\end{equation}

\noindent For $1 \le i \le n$
\begin{eqnarray*}
K_{d_i}(f_i(z))
&=&
\frac{2d_i+1}{\pi^{d_i+1}}\frac{1}{\left(1-\left\lvert f_i(z)\right\rvert^2\right)^{2d_i+2}}
\\
&=&
\frac{2d_i+1}{\pi^{d_i+1}}\frac{1}{(2d_i+1)!}\sum_{m=0}^{\infty} (m+2d_i+1)(m+2d_i)\cdots (m+1) \left\lvert f_i(z)\right\rvert^{2m}
\\
&=&
\frac{2d_i+1}{\pi^{d_i+1}}\left(1+\frac{1}{(2d_i+1)!}\sum_{m=1}^{\infty}(m+2d_i+1)(m+2d_i)\cdots (m+1) \left\lvert f_i(z)\right\rvert^{2m}\right)
\end{eqnarray*}
and similarly
\[
K_d(z)
=
\frac{2d+1}{\pi^{d+1}}\left(1+\frac{1}{(2d+1)!}\sum_{m=1}^{\infty}(m+2d+1)(m+2d)\cdots(m+1) \left\lvert z\right\rvert^{2m}\right).
\]
Therefore, we can write $K_{d_i}(f_i(z))=\frac{2d_i+1}{\pi^{d_i+1}}\,\tilde{K}_{d_i}(f_i(z))$ and $K_d(z)=\frac{2d+1}{\pi^{d+1}}\,\tilde{K}_d(z)$, where
\[
\tilde{K}_{d_i}(f_i(z))=1+\sum_{m=1}^{\infty}\left\lvert c^i_m\,f_i(z)\right\rvert^{2m},
\quad
(c^i_m)^{2m}=\frac{1}{(2d_i+1)!}(m+2d_i+1)(m+2d_i)\cdots (m+1)
\]
and
\[
\tilde{K}_{d}(z)=1+\sum_{m=1}^{\infty} \left\lvert c_m\,z\right\rvert^{2m},
\quad
(c_m)^{2m}=\frac{1}{(2d+1)!}(m+2d+1)(m+2d)\cdots(m+1).
\]
Since $f_i(0)=0$, the Taylor expansions of $\log \tilde{K}_{d_i}(f_i(z))$ and $\log \tilde{K}_{d}(z)$  at $z = 0$ do not contain pure terms. Hence, rewriting (\ref{Ng2}) as
\[
\sum_{i=1}^n \log\frac{2d_i+1}{\pi^{d_i+1}} +  \sum_{i=1}^n\log \tilde{K}_{d_i}(f_i(z))=\log\frac{2d+1}{\pi^{d+1}} + \log \tilde{K}_{d}(z)+ \Re \psi(z),
\]
we see that $\Re \psi$ must be a constant since its Taylor expansion contains only pure terms. In fact,  
\[
\Re \psi \equiv \log\left(\frac{1}{\pi^{\sum d_i+ (n-1)-d }}\,\frac{\prod_i (2d_i+1)}{(2d+1)}\right)
\]
since both $\log \tilde{K}_{d_i}(f_i(z))$ and $\log \tilde{K}_{d}(z)$ vanish at $0$. Hence,
\[
\sum_{i=1}^n\log K_{d_i}(f_i(z))=\log K_{d}(z) + \log\left(\frac{1}{\pi^{\sum d_i+ (n-1)-d }}\,\frac{\prod_i (2d_i+1)}{(2d+1)}\right),\quad z\in \mathbb{D}
\]
which simplifies to 
\begin{equation}\label{Fxn eqn}
\prod_{i=1}^n \left(1-\left\lvert f_i(z)\right\rvert^2\right)^{d_i+1}=\left(1-\left\lvert z\right\rvert^2\right)^{d+1}
\end{equation}
upon substituting the explicit expressions of the kernels. On expanding as a power series and comparing the coefficients of $\left\lvert z\right\rvert^2$ on both sides, we obtain
\[
\sum_{i=1}^n (d_i+1)\left\lvert f_i'(0)\right\rvert^2=(d+1).
\]
An application of the Schwarz lemma to each component $f_i$ shows that $\left\lvert f_i'(0)\right\rvert\leq 1$ and hence,
\begin{equation}
d+1\leq \sum_{i=1}^n (d_i + 1)
\end{equation}
with equality holding precisely when $\left\lvert f_i'(0)\right\rvert=1$ for all $i$, that is, $f_i(z)=\alpha_i\,z$ for unimodular constants $\alpha_i\in\mathbb{C}$.

\medskip

Moving towards the other part of this theorem, we first make some preparations, inspired by Ng's work \cite{Ng}.

\begin{prop}\label{extension of graph}
There exists an irreducible one-dimensional projective analytic subvariety $V$ in $\mathbb{P}^1\times(\mathbb{P}^1)^n$ extending the graph of $f$ such that $\pi: V \ra \mbb P^1$, the natural projection to the first factor, is a finite branched covering. 
\end{prop}

\begin{proof}
Expand both sides of the functional equation
\[
\prod_{i=1}^n \left(1-\lvert f_i\rvert^2\right)^{d_i+1} = (1-\lvert z\rvert^2)^{d+1},
\]
and rearrange terms to write it as
\[
\sum_{j=1}^{m_1}\lvert A_j\rvert^2 = \sum_{j=1}^{m_2} \lvert B_j\rvert^2,
\]
where $A_i=f_i$ for $1\leq i\leq n$. Also, the terms are rearranged in a manner such that 
\begin{enumerate}
\item In the cases when $d_i$, $d$ are odd for all $i$, and when $d_i$, $d$ are even for some $i$, we have
\[
m_1=m_2=\frac{1}{2}\left(\prod\limits_{i=1}^n(d_i+2)+(d+2)\right)\geq 2^{n-1}+1\geq n.
\]
\item Otherwise
\[
m_2+1=m_1=\frac{1}{2}\left(\prod\limits_{i=1}^n(d_i+2)+(d+2)+1\right)\geq n.
\]
In this case, we will add a term $B_{m_1}=0$ on the right-hand side to ensure that both sides have the same number of terms.
\end{enumerate}
Here, $A_j$, $B_j$ are holomorphic functions on $\mathbb{D}$. Therefore, there exists an $m_1\times m_1$ unitary matrix $U$ such that
\[
\begin{pmatrix}
A_1\\
A_2\\
\vdots\\
A_{m_1}
\end{pmatrix}
=
U
\begin{pmatrix}
B_1\\
B_2\\
\vdots\\
B_{m_1}
\end{pmatrix}.
\]
The first $n$ equations are of the form
\[
f_i=Q_i(z,f_1,\ldots,f_n),\quad i=1,\ldots,n,
\]
where $Q_i$'s are polynomials that do not contain linear terms in the $f_i$'s. These $n$ equations define a projective subvariety $W \subset \mathbb{P}^1\times(\mathbb{P}^1)^n$ containing the ($1$-dimensional) graph of $f$. Define the smooth map $H:\mathbb{C}^{n+1}\rightarrow\mathbb{C}^n$ by $H_i(z,f_1,\ldots,f_n)=f_i-Q_i(z,f_1,\ldots,f_n)$.  Since $f$ vanishes at $0$, the point $\bar{0}\in\mathbb{P}^1\times(\mathbb{P}^1)^n$ lies on $W$. Thus, $H(0)=0$ and
\[
dH(0)=\left[
\begin{matrix}
* & 1 & 0 & \cdots & 0\\
* & 0 & 1 & \cdots & 0\\
\vdots& \vdots & \vdots & \ddots & \vdots \\
* & 0 & 0 & \cdots & 1
\end{matrix}
\right].
\]
By the implicit function theorem, there exists a unique holomorphic function $g:U\rightarrow \mathbb{C}^n$ on some neighborhood $U$ of $0$ such that $H(z,g(z))=0$ for all $z\in U$. Therefore, $W$ is smooth at $\bar{0}$ and is one-dimensional near $\bar{0}$. 

\medskip

Let $V$ be the irreducible component of $W$ containing $0$. Note that $V$ contains the graph of $f$. The projection $\pi :V\rightarrow\mathbb{P}^1$ to the first factor is a proper map and hence $\pi(V)$ is an analytic subvariety of $\mathbb{P}^1$. Since $V$ contains the graph of $f$, $\pi(V)$ cannot be discrete. Therefore, $\pi(V)=\mathbb{P}^1$. Also, since $V$ is connected, there exists a positive integer $s$ such that for a generic point $z\in\mathbb{P}^1$, $\#(\pi^{-1}(z))=s$.
\end{proof}

Consider the projections $P_i: V\rightarrow \mathbb{P}^1\times\mathbb{P}^1$ defined by $P_i(z,\zeta_1,\ldots,\zeta_n)=(z,\zeta_i)$. Note that all the $P_i$'s are proper. Therefore, $V_i=P_i(V)\subset\mathbb{P}^1\times\mathbb{P}^1$ is a one-dimensional projective subvariety containing the graph of the component function $f_i:(\mathbb{D},ds^2_d)\rightarrow (\mathbb{D},ds^2_{d_i})$. Let $\pi_i:V_i\rightarrow\mathbb{P}^1$ denote the proper projection map onto the first factor which exhibits $V_i$ as a finite branched cover over $\mbb P^1$. Let $s_i$ be the cardinality of $\pi_i^{-1}(z)$ for a generic $z \in \mbb P^1$. Since $\pi=\pi_i\circ P_i$, it follows that for each $i$, $s_i$ divides $s$. 

\medskip

Lemmas 4.3, 4.4 and Proposition 4.5 from \cite{Ng} can now be applied verbatim to conclude the following. First, if $(z,a_1,\ldots,a_n), (w,b_1,\ldots,b_n)$ be any two points on $V$, then
\begin{equation}\label{Fxn eqn 1}
\prod_{i=1}^n\left(1-a_i\bar{b_i}\right)^{d_i+1}=(1-z\bar{w})^{d+1}.
\end{equation}
Second, let $(z,a_1,\ldots,a_n)$ be a point on $V$. If $z \in\mathbb{C}$ then $a_i \in \mbb C$ for all $i$. Finally, fix an $i \le n$. Then $V_i\subset\mathbb{P}^1\times\mathbb{P}^1$ is defined by 
\begin{equation}\label{V_i}
h^{s_i}+P_{s_i-1}(z)h^{s_i-1}+P_{s_i-2}(z)h^{s_i-2}+\cdots+P_1(z)h+P_0(z)=0,
\end{equation}
where for all $j$, $P_j(z)=a_jz+b_j$ for some $a_j,b_j\in\mathbb{C}$ and $b_0=0$. In fact, $(\ref{V_i})$ is invariant under the assignment $(z, h) \mapsto (1/\overline z, 1/\overline h)$ and this implies that for $\lvert z\rvert=1$, if $h_j$, $1\leq j\leq s_i$, are the $s_i$ roots of (\ref{V_i}), then $1/\overline{h_j}$, $1\leq j\leq s_i$, is a permutation of the roots. As a corollary, by replacing the pair $(z, h)$ by $(1/\overline z, 1/\overline h)$, clearing denominators and taking conjugates in (\ref{V_i}), we get
\[
z = -h \left( \frac{\ov a_{s_i-1} + \ov a_{s_i-2} h + \ldots + \ov a_{0} h^{s_i - 1}}{1 + \ov b_{s_i-1} h + \ldots + \ov b_{1} h^{s_i - 1}} \right)
\]
and this implies that for $h = f_i(z)$, $z = R_i(f_i(z))$, where $R_i$ is the rational function of degree $s_i$ as given above.

\medskip

Using Calabi's theorem, it can be shown that every germ 
\[
g:(\mathbb{D},ds^2_d;x_0)\rightarrow (\mathbb{D},ds^2_{d_1};y_{01})\times\cdots\times(\mathbb{D},ds^2_{d_n};y_{0n})
\]
of a holomorphic isometry can be extended over the whole unit disc (ref. Appendix). Let $A\subset\mathbb{P}^1$ be a discrete set such that $\pi$, when restricted to $V\setminus \pi^{-1}(A)$, is a holomorphic covering. For $(z_0,a_1,\ldots, a_n)\in V\setminus\pi^{-1}(A)$ with $\vert z_0\vert\neq 1$, there exists a germ of the graph of a holomorphic function at $z_0$ that takes value $(a_1,\ldots,a_n)$ at $z_0$ and lies in $V$. Since $\mathbb{P}^1\setminus\overline{\mathbb{D}}$ is biholomorphic to the unit disc, it can be shown that every such germ must extend to the whole of $\mathbb{D}$ or $\mathbb{P}^1\setminus\ov{\mathbb{D}}$ depending upon whether $z_0\in\mathbb{D}$ or $z_0\in\mathbb{P}^1\setminus\ov{\mathbb{D}}$. All these observations have been made in \cite{Ng}, but since we are dealing with the weighted case, Section \ref{Appendix} contains a brief summary of the minor modifications need to be made. Therefore, the branch points of $f$ lie on the unit circle. Since every branch point of $f_i$ must be a branch point of $\pi:V\rightarrow\mathbb{P}^1$, the branch points of $f_i$ must lie on the unit circle.

\medskip

Recall that each $s_i$ is a factor of $s$. Within $s$ branches of $f$, each coordinate function $f_i$ has $s_i$ distinct branches and each distinct branch repeats itself $s/s_i$ times. Further, each $f_i$ takes infinite value at $\infty$ on at least one branch among $s_i$ branches. Since $z = \infty$ is not a branch point, there is exactly one branch among the $s_i$ branches such that $f_i(\infty)=\infty$. Thus, each $f_i$ takes an infinite value $s/s_i$ times at $\infty$ within $s$ branches of $f$. 
Recall the functional equation
\[
\prod_{i=1}^n \left(1-\left\lvert f_i(z)\right\rvert^2\right)^{d_i+1}=\left(1-\left\lvert z\right\rvert^2\right)^{d+1} .
\]
Therefore, on each branch of $f$, at least one component function takes an infinite value at $\infty$ and sum of the multiplicities with which $\infty$ is taken by the component functions should equal $2d+2$. Since each $f_i$ must take an infinite value at $z = \infty$ on at least one of the $s$ branches of $f$, it follows that for every $1\leq i\leq n$, there exist $1\leq i=k_1,\ldots,k_{j_i}\leq n$ such that
\begin{equation}
\sum_{r=1}^{j_i} (2d_{k_r}+2)=2d+2.
\end{equation}
In particular, $d_i\leq d$ for every $1\leq i\leq n$.
\end{proof}


\section{Proof of Theorem \ref{1.6}}

Recall Mok's $n$-th root embedding from \cite{Mok1} -- let $\mathbb{H}$ denote the upper half-plane equipped with the hyperbolic metric $ds^2_{\mbb H}$. For $n\geq 2$, define $g:(\mathbb{H},ds^2_{\mathbb{H}})\rightarrow (\mathbb{H},ds^2_{\mathbb{H}})\times\cdots\times(\mathbb{H},ds^2_{\mathbb{H}})$ by
\[
g(z)=(z^{1/n},\gamma z^{1/n},\cdots,\gamma^{n-1}z^{1/n}),
\]
where $\gamma=e^{i\pi/n}$, and if $z=re^{i\theta}$ with $r>0$ and $0<\theta<\pi$ then $z^{1/n}=r^{1/n}e^{i\theta/n}$. Then, $g$ is a holomorphic isometric embedding which is called the $n$-th root embedding. 

\medskip

A complete characterization of all holomorphic isometries from $(\mbb D, ds^2_d)$ to $(\mbb D, ds^2_{d_1}) \times (\mbb D, ds^2_{d_2})$ for non-negative integers $d,d_1$ and $d_2$ was obtained by Chan, Xiao and Yuan in \cite{CXY}. Motivated from it, we are interested in examining isometries from $\mbb D$ into $\mbb D \times \mbb D \times \mbb D$. Let 
$f =(f_1,f_2,f_3): (\mbb D, ds^2_d) \ra (\mbb D, ds^2_{d_1}) \times (\mbb D, ds^2_{d_2}) \times (\mbb D, ds^2_{d_3})$
be a holomorphic isometry, where $d,d_1,d_2$ and $d_3$ are non-negative integers. If any of the component functions of $f$ is constant, then our isometry reduces to either an isometry between a pair of discs or an isometry from a disc to a bidisc. Both cases are known. Therefore, assume that all the component functions of $f$ are non-constant.

\begin{prop}
$f$ extends continuously from $\overline{\mathbb{D}}$ to $\overline{\mathbb{D}}^3$. 
\end{prop}

\begin{proof}
It is enough to show that each $f_i$ extends continuously from $\overline{\mathbb{D}}$ to $\overline{\mathbb{D}}$. Fix $i\in\{1,2,3\}$. For $z_0\in\partial\mathbb{D}$, let $\pi^{-1}_i(z_0)=\{(z_0,a_1),\ldots,(z_0,a_r)\}$ where $r\leq s_i$. Define $S_i$ to be the set of all the limit points of sequences $\{f_i(z_k)\}$ where $z_k\rightarrow z_0$ as $k\rightarrow\infty$. Then, $S_i\subset \{a_1,\ldots,a_r\}$. But $S_i$, being the cluster set of $z_0$ under $f_i$, must be connected and hence must be a point. Therefore, $f_i$ extends continuously from $\ov{\mathbb{D}}$ to $\ov{\mathbb{D}}$.
\end{proof}

For $i=1,2,3$, define
\[
A_i=\{b\in\partial\mathbb{D}: b \text{ is not a branch point of } f \text{ and }\left\lvert f_i(z)\right\rvert\rightarrow 1 \text{ as }z\rightarrow b\}.
\]

\begin{lem}\label{Open A_i}
Fix $i\in\{1,2,3\}$. Assume that $b\in\partial\mathbb{D}$ is not a branch point of $f$ and $\vert f_i(z)\vert\rightarrow 1$ as $z\rightarrow b$. There exists a neighborhood $U_b\subset\mathbb{C}$ of $b$ such that every $b'\in U_b\cap\partial\mathbb{D}$ is not a branch point of $f$, $\vert f_i(z)\vert\rightarrow 1$ as $z\rightarrow b'$ and the vanishing order of $1-f_i(z)\overline{f_i(b')}$ at $z=b'$ equals to one. In particular, $A_i$ is open.
\end{lem}

\begin{proof}
We can assume, without loss of generality, that $f(0)=0$, as the statement remains true after a reparametrization. 

\medskip

Since every branch point of $f_i$ is a branch point of $f$, $b$ is not a branch point of $f_i$. Let, $\pi_i^{-1}(b)=\{(b,a_1),\ldots,(b,a_{s_i})\}$. Recall that $\pi_i\vert_{\pi_i^{-1}(\mathbb{P}^1\setminus A)}$ is a holomorphic covering map. Therefore, there exists injective holomorphic function $h_j$, $1\leq j\leq s_i$ on a neighborhood $W\subset\mathbb{C}$ of $b$ such that $h_j(b)=a_j$ and $(z,h_j(z))\in V_i$ for all $z\in W$ and $1\leq j\leq s_i$. Choose $W$ (shrink, if necessary) so that the closures of $W_j=h_j(W)$ are mutually disjoint, and $1/\bar{z}\in W$ whenever $z\in W$. 
There exists $1\leq j_i\leq s_i$ such that $h_{j_i}\equiv f_i$ on $W\cap\mathbb{D}$. Without loss of generality, let $j_i=1$.

\medskip

Put $\mathbb{D}^+=\mathbb{D}$ and $\mathbb{D}^-=\mathbb{P}^1\setminus\ov{\mathbb{D}}$. Recall that $(1/\bar{z},1/\bar{h})\in V_i$ whenever $(z,h)\in V_i$.
Therefore, $U^+=h_1(W\cap\mathbb{D}^+)$ and $U^-=h_1(W\cap\mathbb{D}^-)$ are open and disjoint sets lying in $\mathbb{D}^+$ and $\mathbb{D}^-$ respectively. Since $\vert h_1(b)\vert=1$, it follows from Open mapping theorem that $h_1(W)$ consists of an open arc $\gamma$ of $\partial\mathbb{D}$ with $h_1(b)\in\gamma$. Now $h_1$ is invertible, and the functional equation gives $h_1^{-1}(\gamma)\subset\partial\mathbb{D}$.
Hence, there exists an open neighborhood $U_b\subset\mathbb{C}$ of $b$ such that $\vert f_i(z)\vert\rightarrow 1$ as $z\rightarrow b'$ for any $b'\in U_b\cap \partial \mathbb{D}$. 

\medskip

Further, shrink $U_b$, if necessary, so that every $b'\in U_b\cap\partial\mathbb{D}$ is not a branch point of $f$ and $f_i'(b')\neq 0$. The latter ensures that the vanishing order of $1-f_i(z)\overline{f_i(b')}$ at $z=b'$ equals to one.
\end{proof}

We can assume, without loss of generality, that $f(0)=0$, as we shall characterize $f$ up to reparametrization.

\medskip

Since $f(0)=0$, we have the functional equation 
\[
\prod_{i=1}^3 (1-\vert f_i(z)\vert^2)^{d_i+1}=(1-\vert z\vert^2)^{d+1}.
\]
Therefore, there must exist some $i$ such that $\vert f_i(z)\vert\rightarrow 1$ as $\vert z\vert \rightarrow 1$. Since there exist only finitely many branch points of $f$ on $\partial\mathbb{D}$, at least one $A_i$ must be non-empty. The proof now divides into the following cases.

\begin{proof}[Case 1.] $A_1\cap A_2\cap A_3\neq \emptyset$.

\medskip

That is, there is a point $b\in\partial \mathbb{D}$ which is not a branch point of $f$ and $\vert f_i(z)\vert\rightarrow 1$ as $z\rightarrow b$ for all $i=1,2,3$. Since $f(0)=0$, we have
\[
\prod_{i=1}^3\left(1-f_i(z)\overline{f_i(b)}\right)^{d_i+1}=\left(1-z\bar{b}\right)^{d+1}.
\]
Note that the vanishing order of $1-f_i(z)\overline{f_i(b)}$ at $z=b$ equals to one for all $i=1,2,3$. 
On comparing the vanishing order at $b$ in the equation, we obtain $\sum_{i=1}^3(d_i+1)=d+1$. Therefore, by Theorem \ref{1.5}, there exist unimodular constants $\alpha_1,\alpha_2$ and $\alpha_3$ such that $f(z)=(\alpha_1 z, \alpha_2 z,\alpha_3 z)$, after a possible reparametrization.
\end{proof}

\begin{proof}[Case 2.] Each $A_i$ is non-empty, and all $A_i$ are mutually disjoint.

\medskip

That is, there are three distinct points $b_1,b_2$ and $b_3$ on $\partial\mathbb{D}$ which are not branch points of $f$ such that for every $i=1,2,3$
\[
\vert f_i(z)\vert \rightarrow 1\quad\text{and}\quad \vert f_j(z)\vert\rightarrow r_j^i<1\quad \text{for } j\neq i\quad \text{as}\quad z\rightarrow b_i.
\]
Since $f(0)=0$, we have
\[
\left(1-f_1(z)\overline{f_1(b_i)}\right)^{d_1+1}\left(1-f_2(z)\overline{f_2(b_i)}\right)^{d_2+1}\left(1-f_3(z)\overline{f_3(b_i)}\right)^{d_3+1}=\left(1-z\bar{b_i}\right)^{d+1}.
\]
Note that the vanishing order of $1-f_i(z)\overline{f_i(b_i)}$ at $z=b_i$ equals to one for all $i=1,2,3$.
On comparing the vanishing order at $b_i$ in the equation, we obtain $d_1=d_2=d_3=d$. Therefore, $f$ is an isometry from $(\mathbb{D},ds^2_{\mathbb{D}})$ to $(\mathbb{D},ds^2_{\mathbb{D}})^3$. By Ng \cite{Ng}, it follows that, up to reparametrization, $f$ is either the cube root embedding or $f$ can be factorized as 
\[
f(z)=\left(\alpha_1(z),\alpha_2(\beta_1(z)),\beta_2(\beta_1(z))\right),
\]
where $z\mapsto (\alpha_i(z),\beta_i(z))$ is a holomorphic isometry from $(\mathbb{D},ds^2_{\mathbb{D}})$ to $(\mathbb{D},ds^2_{\mathbb{D}})^2$ for $i=1,2$. 

\medskip

Conversely, if up to a reparametrization, $f$ is the cube root embedding, then each $A_i$ is non-empty and all $A_i$ are mutually disjoint. Let us now look at the latter possibility. Assume that $f=(\alpha_1,\alpha_2\circ\beta_1,\beta_2\circ\beta_1)$.
For $i=1,2$, define
\[
\Gamma_{\alpha_i}=\{b\in\partial\mathbb{D}: b \text{ is not a branch point of } z\mapsto (\alpha_i(z),\beta_i(z)) \text{ and }\left\lvert \alpha_i(z)\right\rvert\rightarrow 1 \text{ as }z\rightarrow b\},
\]
and similarly
\[
\Gamma_{\beta_i}=\{b\in\partial\mathbb{D}: b \text{ is not a branch point of } z\mapsto (\alpha_i(z),\beta_i(z)) \text{ and }\left\lvert \beta_i(z)\right\rvert\rightarrow 1 \text{ as }z\rightarrow b\}.
\]
By Ng \cite{Ng}, every holomorphic isometry from $(\mathbb{D},ds^2_{\mathbb{D}})$ to $(\mathbb{D},ds^2_{\mathbb{D}})^2$ is, up to reparametrizations, a square root embedding. Therefore, $\Gamma_{\alpha_i}$ and $\Gamma_{\beta_i}$ are non-empty disjoint subsets of $\partial\mathbb{D}$ for both $i=1,2$.  So, $A_1$ is non-empty and all $A_i$ are mutually disjoint.
For both $A_2$, $A_3$ to be non-empty, $\beta_1(\Gamma_{\beta_1})$ must have non-empty intersection with both $\Gamma_{\alpha_2}$ and $\Gamma_{\beta_2}$.

Therefore, if we take any reparametrization of $(\alpha_1,\alpha_2\circ\beta_1,\beta_2\circ\beta_1)$, then all $A_i$ are mutually disjoint but for each $A_i$ to be non-empty,  $\beta_1(\Gamma_{\beta_1})$ must have non-empty intersection with both $\Gamma_{\alpha_2}$ and $\Gamma_{\beta_2}$.
\begin{exam}\label{example}
Let $g=(g_1,g_2):(\mathbb{H},ds^2_{\mathbb{H}})\rightarrow (\mathbb{H},ds^2_{\mathbb{H}})^2$ be the square root embedding defined by
\[
g(z)= (z^{1/2},e^{i\pi/2} z^{1/2}),
\]
where if $z=r e^{i\theta}$ with $r>0$ and $0<\theta<\pi$ then $z^{1/2}=r^{1/2}e^{i\theta/2}$. The points $0$ and $\infty$ are the branch points of $g$. As $z\rightarrow z_0\in\partial\mathbb{H}$, where $z_0$ is not a branch point of $g$, 
\[
g_1(z)\rightarrow \partial\mathbb{H}=\mathbb{R} \,\text{ only if }\, z_0\in\mathbb{R}^+ \quad\text{and}\quad g_2(z)\rightarrow \partial\mathbb{H}=\mathbb{R} \,\text{ only if }\, z_0\in\mathbb{R}^-.
\]
Also, $g_1(\mathbb{R}^+)=\mathbb{R}^+$ and $g_2(\mathbb{R}^-)=\mathbb{R}^-$.
Let $\varphi:\mathbb{H}\rightarrow\mathbb{D}$ be the biholomorphism given by
\[
\varphi(z)=\frac{z-i}{1-iz},\quad z\in\mathbb{H}.
\]
The map $G=(G_1,G_2)=(\varphi,\varphi)\circ g\circ \varphi^{-1}:(\mathbb{D},ds^2_{\mathbb{D}})\rightarrow (\mathbb{D},ds^2_{\mathbb{D}})^2$ is a holomorphic isometry. Let $\mathcal{H}^+$ and $\mathcal{H}^-$ denote the right and left half planes respectively. Then, $\varphi(\mathbb{R}^+)=\partial\mathbb{D}\cap \mathcal{H}^+$ and $\varphi(\mathbb{R}^-)=\partial\mathbb{D}\cap \mathcal{H}^-$. Therefore, $-i$ and $i$ are the branch points of $G$. And, for every $z_0\in\partial\mathbb{D}$ which is not a branch point of $G$, we have as $z\rightarrow z_0\in\partial\mathbb{D}$ that
\[
G_1(z)\rightarrow \partial\mathbb{D}\,\text{ only if }\, z_0\in \partial\mathbb{D}\cap\mathcal{H}^+ \quad\text{and}\quad G_2(z)\rightarrow \partial\mathbb{D}\,\text{ only if }\, z_0\in \partial\mathbb{D}\cap\mathcal{H}^-.
\]
Also, $G_1(\partial\mathbb{D}\cap\mathcal{H}^+)=\partial\mathbb{D}\cap\mathcal{H}^+$ and $G_2(\partial\mathbb{D}\cap\mathcal{H}^-)=\partial\mathbb{D}\cap\mathcal{H}^-$. 
Take $(\alpha_1,\beta_1)=G$ and $(\alpha_2,\beta_2)=G\circ (e^{i\pi/2})$. Here, $\beta_1(\Gamma_{\beta_1})=\Gamma_{\beta_1}=\partial\mathbb{D}\cap\mathcal{H}^-$, $\Gamma_{\alpha_2}=e^{-i\pi/2}\,\Gamma_{\alpha_1}=e^{-i\pi/2}(\partial\mathbb{D}\cap\mathcal{H}^+)$ and $\Gamma_{\beta_2}=e^{-i\pi/2}\,\Gamma_{\beta_1}=e^{-i\pi/2}(\partial\mathbb{D}\cap\mathcal{H}^-)$. Therefore, $\beta_1(\Gamma_{\beta_1})$ has non-empty intersection with both $\Gamma_{\alpha_2}$ and $\Gamma_{\beta_2}$.
\end{exam}
Hence, up to reparametrization, $f$ is either the cube root embedding or $f$ can be factorized as 
\[
f(z)=\left(\alpha_1(z),\alpha_2(\beta_1(z)),\beta_2(\beta_1(z))\right)
\]
where $z\mapsto (\alpha_i(z),\beta_i(z))$ is a holomorphic isometry from $(\mathbb{D},ds^2_{\mathbb{D}})$ to $(\mathbb{D},ds^2_{\mathbb{D}})^2$ for both $i=1,2$ such that $\beta_1(\Gamma_{\beta_1})$ has non-empty intersection with both $\Gamma_{\alpha_2}$ and $\Gamma_{\beta_2}$.
\end{proof}

\begin{proof}[Case 3.] Exactly one of the sets $A_i$ is empty, and the other two are mutually disjoint. 

\medskip

Without loss of generality, assume that $A_1\cap A_2=\emptyset=A_3$, but $A_1,A_2\neq\emptyset$.
That is, there are two distinct points $b_1,b_2\in\partial \mathbb{D}$ which are not branch points of $f$ such that for $j=1,2$ and $i\neq j$
\[
\vert f_j(z)\vert \rightarrow 1 \quad\text{and}\quad \vert f_i(z)\vert\rightarrow r^i_j<1 \quad \text{as}\quad z\rightarrow b_j.
\]
Since $f(0)=0$, we have
\[
\left(1-f_1(z)\overline{f_1(b_j)}\right)^{d_1+1}\left(1-f_2(z)\overline{f_2(b_j)}\right)^{d_2+1}\left(1-f_3(z)\overline{f_3(b_j)}\right)^{d_3+1}=\left(1-z\bar{b_j}\right)^{d+1}.
\]
Note that the vanishing order of $1-f_j(z)\overline{f_j(b_j)}$ at $z=b_j$ equals to one for both $j=1,2$. On comparing the vanishing order at $b_j$ in the equation,
we obtain $d_1=d_2=d$. It, therefore, follows from Theorem \ref{1.5} that $d_3=d$.
Hence, $f$ is an isometry from $(\mathbb{D},ds^2_{\mathbb{D}})$ to $(\mathbb{D},ds^2_{\mathbb{D}})^3$. By Ng \cite{Ng}, we conclude that, up to reparametrizations, $f$ is either the cube root embedding or $f$ can be factorized as 
\[
f(z)=\left(\alpha_1(z),\alpha_2(\beta_1(z)),\beta_2(\beta_1(z))\right)
\]
where $z\mapsto (\alpha_i(z),\beta_i(z))$ is a holomorphic isometry from $(\mathbb{D},ds^2_{\mathbb{D}})$ to $(\mathbb{D},ds^2_{\mathbb{D}})^2$ for $i=1,2$. 

\medskip

For every reparametrization of the cube root embedding each $A_i$ is non-empty. This rules out the first possibility as one of the sets $A_i$ must be empty.

\medskip

Let us now look at the other possibility. Assume that $f=(\alpha_1,\alpha_2\circ\beta_1,\beta_2\circ\beta_1)$. All $A_i$ are mutually disjoint and $A_1\neq \emptyset$. For $A_3=\emptyset$, we must have $\beta_1(\Gamma_{\beta_1})\cap\Gamma_{\beta_2}=\emptyset$. Since $\Gamma_{\alpha_2}\cup\Gamma_{\beta_2}$ is $\partial\mathbb{D}$ minus finitely many points (branch points), we must have $\beta_1(\Gamma_{\beta_1})\cap \Gamma_{\alpha_2}\neq \emptyset$ and therefore $A_2\neq \emptyset$.

\medskip

Therefore, if we take any reparametrization of $(\alpha_1,\alpha_2\circ\beta_1,\beta_2\circ\beta_1)$, then all $A_i$ are mutually disjoint and at least two of the sets $A_i$ are non-empty. For one of the sets $A_i$ to be empty, we must have that either $\beta_1(\Gamma_{\beta_1})\cap\Gamma_{\alpha_2}=\emptyset$ or $\beta_1(\Gamma_{\beta_1})\cap\Gamma_{\beta_2}=\emptyset$. 
\begin{exam}
Take $(\alpha_1,\beta_1)=G$ and $(\alpha_2,\beta_2)=G\circ (e^{i\pi})$, where $G$ is as in Example \ref{example}. Here, $\beta_1(\Gamma_{\beta_1})=\Gamma_{\beta_1}=\partial\mathbb{D}\cap\mathcal{H}^-$,
$\Gamma_{\alpha_2}=e^{-i\pi}\,\Gamma_{\alpha_1}=\partial\mathbb{D}\cap\mathcal{H}^-$
and 
$\Gamma_{\beta_2}=e^{-i\pi}\,\Gamma_{\beta_1}=\partial\mathbb{D}\cap\mathcal{H}^+$. 
Therefore, $\beta_1(\Gamma_{\beta_1})\cap \Gamma_{\beta_2}=\emptyset$.
\end{exam}
Hence, up to reparametrizations,
\[
f(z)=\left(\alpha_1(z),\alpha_2(\beta_1(z)),\beta_2(\beta_1(z))\right),
\]
where for both $i=1,2$, the map
$
z\mapsto (\alpha_i(z),\beta_i(z))
$
is an isometry from $(\mathbb{D},ds^2_{\mathbb{D}})$ to $(\mathbb{D},ds^2_{\mathbb{D}})^2$ such that either $\beta_1(\Gamma_{\beta_1})\cap\Gamma_{\alpha_2}=\emptyset$ or $\beta_1(\Gamma_{\beta_1})\cap\Gamma_{\beta_2}=\emptyset$.
\end{proof}

\begin{proof}[Case 4.] $A_1\cap A_2\cap A_3=\emptyset$, but all the sets $A_1\cap A_2$, $A_2\cap A_3$ and $A_3\cap A_1$ are non-empty.

\medskip

That is, there are three distinct points $b_1,b_2$ and $b_3$ on $\partial\mathbb{D}$ which are not the branch points of $f$ such that for every $j=1,2,3$ and $i\neq j$
\[
\vert f_i(z)\vert \rightarrow 1\quad\text{and}\quad \vert f_j(z)\vert\rightarrow r_j<1\quad \text{as}\quad z\rightarrow b_j.
\]
Then we have
\[
\left(1-f_1(z)\overline{f_1(b_j)}\right)^{d_1+1}\left(1-f_2(z)\overline{f_2(b_j)}\right)^{d_2+1}\left(1-f_3(z)\overline{f_3(b_j)}\right)^{d_3+1}=\left(1-z\bar{b_j}\right)^{d+1}.
\]
The vanishing order of $1-f_i(z)\overline{f_i(b_j)}$ at $z=b_j$ equals to one for all $i\neq j$. On comparing the vanishing order at $b_j$ in the equation,
we obtain that $\sum\limits_{i=1,2}(d_i+1)=\sum\limits_{i=2,3}(d_i+1)=\sum\limits_{i=1,3}(d_i+1)=d+1$. That is,
\[
d_1=d_2=d_3=r\quad \text{and} \quad d=2r+1
\]
for some non-negative integer $r$. So, the functional equation becomes
\[
\prod_{i=1}^3 \left( 1-\vert f_i(z)\vert^2\right)=(1-\vert z\vert^2)^2.
\]
Therefore, $f$ is an isometry from $(\mathbb{D},ds^2_{\mathbb{D}})$ to $(\mathbb{D},ds^2_{\mathbb{D}})^3$ with the isometric constant $k=2$ (see \cite{Ng}). By Ng \cite{Ng}, it follows that, up to reparametrizations, 
\[
f(z)=(z,\alpha(z),\beta(z)),
\]
where $z\mapsto (\alpha(z),\beta(z))$ is an isometry from $(\mathbb{D},ds^2_{\mathbb{D}})$ to $(\mathbb{D},ds^2_{\mathbb{D}})^2$. Since  $\Gamma_{\alpha}\cap \Gamma_{\beta}=\emptyset$, at least one of the sets $A_i\cap A_j$ is empty. Hence, a contradiction. Therefore, this case is not possible.
\end{proof}

\begin{proof}[Case 5.] Exactly two pairs (out of the possible three) have non-empty intersection.  

\medskip

Without loss of generality, assume that $A_1\cap A_3\neq \emptyset\neq A_2\cap A_3$, but $A_1\cap A_2=\emptyset$.
That is, there are two distinct points $b_1,b_2\in\partial \mathbb{D}$ which are not the branch points of $f$ such that for $j=1,2$ and $i\neq j$
\[
\vert f_i(z)\vert \rightarrow 1 \quad\text{and}\quad \vert f_j(z)\vert\rightarrow r_j<1 \quad \text{as}\quad z\rightarrow b_j.
\]
Then we have
\[
\left(1-f_1(z)\overline{f_1(b_j)}\right)^{d_1+1}\left(1-f_2(z)\overline{f_2(b_j)}\right)^{d_2+1}\left(1-f_3(z)\overline{f_3(b_j)}\right)^{d_3+1}=\left(1-z\bar{b_j}\right)^{d+1}.
\]
The vanishing order of $1-f_i(z)\overline{f_i(b_j)}$ at $z=b_j$ equals to one for all $i\neq j$ and $j=1,2$. On comparing the vanishing order at $b_j$ in the equation,
we obtain 
\[
d+1=\sum\limits_{i=1,3}(d_i+1)=\sum\limits_{i=2,3}(d_i+1).
\]
Therefore, $d_1=d_2$.
Let if possible, $A_1\setminus A_3\neq \emptyset$. Since $A_1\cap A_2=\emptyset$, we have $d_1=d$. But then $d_3=-1$, which is a contradiction. Similarly, we cannot have $A_2\setminus A_3\neq \emptyset$. Thus, both the sets $A_1$ and $A_2$ are contained in $A_3$. 
Note that we also have $A_3\setminus (A_1\cup A_2)= \emptyset$. Otherwise $d_3=d$ and therefore $d_1=d_2=-1$, hence a contradiction. Therefore, for every $b\in\partial\mathbb{D}$ which is not a branch point of $f$, we have
\[
\vert f_3(z)\vert \rightarrow 1\quad\text{as}\quad z\rightarrow b.
\]
Since $f_3$ is a continuous function from $\overline{\mathbb{D}}$ to $\overline{\mathbb{D}}$ and $f$ has only finitely many branch points, $\vert f_3(z)\vert\rightarrow 1$ whenever $z\rightarrow b\in\partial\mathbb{D}$. Thus, $f_3$ is a proper holomorphic map and hence a finite Blaschke product. As noted earlier, $f_3$ cannot have poles in $\mathbb{C}$. Therefore, $f_3(z)=cz^k$ for some $\vert c\vert=1$ and $k \ge 1$. Further, there exists a rational function $R$ such that $R(f_3(z))=R(c z^k)=z$. Let $m$ be the degree of $R$. Then the degree of $R(c z^k)$ equals $mk$. Hence, $mk=1$. So, $k=1$ (and $m=1$).
Thus, the functional equation becomes
\[
\prod_{i=1}^2\left(1-\vert f_i(z)\vert^2\right)^{d_1+1}(1-\vert z\vert^{2})^{d_3+1} =(1-\vert z\vert^2)^{d_1+d_3+2},
\]
which simplifies as
\[
\prod_{i=1}^2\left(1-\vert f_i(z)\vert^2\right)=(1-\vert z\vert^2).
\]
Therefore, the map $(f_1,f_2):(\mathbb{D},ds^2_{\mbb{D}})\rightarrow(\mathbb{D},ds^2_{\mbb{D}})^2$ is a holomorphic isometry and $f_3(z)=cz$ for some $\vert c\vert=1$. Note that, we indeed have $A_1\cap A_3\neq \emptyset\neq A_2\cap A_3$ and $A_1\cap A_2=\emptyset$.

\medskip

Hence,  
\[
d+1=\sum\limits_{i=1,3}(d_{k_i}+1)=\sum\limits_{i=2,3}(d_{k_i}+1),
\]
where $k_i\neq k_j$ for $i\neq j$ and $k_i\in\{1,2,3\}$. Up to a reparametrization, we have 
\[
f(z)=(z,\alpha(z),\beta(z)),
\]
where $z\mapsto (\alpha(z),\beta(z))$ is an isometry from $(\mathbb{D},ds^2_{\mbb D})$ to $(\mathbb{D},ds^2_{\mbb D})^2$.
\end{proof}

\begin{proof}[Case 6.] All $A_i$ are non-empty, and exactly one pair (out of three) has a non-empty intersection. 

\medskip

Without loss of generality, let $A_1\cap A_2\neq \emptyset\neq A_3$, but $A_1\cap A_3=\emptyset=A_2\cap A_3$.
That is, there are two distinct points $b_1,b_2\in\partial \mathbb{D}$ which are not the branch points of $f$ such that for $i=1,2$, we have
\begin{enumerate}
\item[(i)] $\vert f_i(z)\vert \rightarrow 1$ and $\vert f_3(z)\vert\rightarrow r_1<1$ as $z\rightarrow b_1$,
\item[(ii)] $\vert f_i(z)\vert \rightarrow r^i_2<1$ and $\vert f_3(z)\vert\rightarrow 1$ as $z\rightarrow b_2$.
\end{enumerate}
Since $f(0)=0$, we have
\[
\left(1-f_1(z)\overline{f_1(b_j)}\right)^{d_1+1}\left(1-f_2(z)\overline{f_2(b_j)}\right)^{d_2+1}\left(1-f_3(z)\overline{f_3(b_j)}\right)^{d_3+1}=\left(1-z\bar{b_j}\right)^{d+1}.
\]
Note that for $i=1,2$, the vanishing orders of $1-f_i(z)\overline{f_i(b_1)}$ at $z=b_1$ and $1-f_3(z)\overline{f_3(b_2)}$ at $z=b_2$ are equal to one. On comparing the vanishing orders at $b_1$ and $b_2$ in the equation,
we obtain
\begin{equation}\label{relation}
\sum\limits_{i=1,2}(d_i+1)=d+1\quad\text{and}\quad d_3=d.
\end{equation}
As noted before, since
\[
\prod_{i=1}^3 (1-\vert f_i(z)\vert^2)^{d_i+1} = (1-\vert z\vert^2)^{d+1},
\]
at least one component function of $f$ takes an infinite value at $\infty$ on each branch of $f$ and sum of the multiplicities with which $\infty$ is taken by the component functions should equal $2d+2$. Because of the relation $(\ref{relation})$, there are only two possibilities on each branch of $f$. First, both $f_1$ and $f_2$ take an infinite value but $f_3$ takes a finite value at $\infty$. Second, only $f_3$ takes an infinite value at $\infty$. 
As noted in the proof of Theorem \ref{1.5}, each $f_i$ has $s_i$ distinct branches and each distinct branch repeats itself $s/s_i$ times. Further, each $f_i$ must take an infinite value at $\infty$ on exactly one of the distinct $s_i$ branches of $f_i$. 

\medskip

Since both $f_1$ and $f_2$ take an infinite value at $\infty$ together, we must have $s_1 = s_2$. Therefore, $f_1$ and $f_2$ take an infinite value on $s/s_1$ branches of $f$, and $f_3$ takes an infinite value at $\infty$ on $s/s_3$ branches of $f$. Hence, we conclude that
\[
s = \frac{s}{s_1} + \frac{s}{s_3}, \quad\text{that is,}\quad 1 = \frac{1}{s_1} + \frac{1}{s_3}
\]
Since $s_1$ and $s_3$ are positive integers, the only possible value for $s_1$ and $s_3$ is $2$. Thus, $s_1 = s_2 = s_3=2$. 

\medskip

Since the sheeting number of $f_3$ is $2$, there exists an isometry $z\mapsto (\alpha(z),\beta(z))$ from $(\mathbb{D}, ds^2_{\mathbb{D}})$ to $(\mathbb{D}, ds^2_{\mathbb{D}})^2$ such that $f_3 = \beta$ (see \cite{Ng}). Also, $(1-\vert \alpha(z)\vert^2)(1-\vert \beta(z) \vert^2) = (1 -\vert z\vert^2)$. Therefore,
\[
(1-\vert\beta(z)\vert^2)^{d+1}\,\prod_{i=1}^2 (1-\vert f_i(z)\vert^2)^{d_i+1} = (1-\vert \alpha(z)\vert^2)^{d+1}(1-\vert \beta(z) \vert^2)^{d+1}.
\]
That is,
\[
\prod_{i=1}^2 (1-\vert f_i(z)\vert^2)^{d_i+1} = (1-\vert \alpha(z) \vert^2)^{d+1}.
\]
Since $0$ is not a branch point, there exists a local inverse of $\alpha$ on a neighborhood $U\subset\mathbb{D}$ of $0$. So,
\[
\prod_{i=1}^2 (1-\vert (f_i\circ \alpha^{-1})(z)\vert^2)^{d_i+1} = (1-\vert z \vert^2)^{d+1}, \quad z\in U
\]
where $(d_1+1) + (d_2 +1) = d+1$. Therefore, $z\mapsto (f_1\circ \alpha^{-1},f_2\circ\alpha^{-1})$ is a germ of a holomorphic isometry from $(\mathbb{D},ds^2_{d};0)$ to $(\mathbb{D},ds^2_{d_1};0)\times(\mathbb{D},ds^2_{d_2};0)$. This germ extends to a holomorphic isometry to the entire unit disc and maps it into the bidisc (see Section \ref{Appendix}). Since $(d_1+1) + (d_2 +1) = d+1$, we therefore conclude from Case $1$ that $(f_i\circ \alpha^{-1})(z) = c_i z$ for some unimodular constants $c_i\in\mathbb{C}$. Thus, we have proved that
\[
(f_1,f_2,f_3) = (c_1\alpha, c_2\alpha,\beta)
\]
for some isometry $z\mapsto (\alpha(z),\beta(z))$ from $(\mathbb{D}, ds^2_{\mathbb{D}})$ to $(\mathbb{D}, ds^2_{\mathbb{D}})^2$ and unimodular constants $c_i\in\mathbb{C}$. Note that, we indeed have $A_1\cap A_2\neq \emptyset\neq A_3$ and $A_1\cap A_3=\emptyset=A_2\cap A_3$.

\medskip

Hence,
\[
d_{k_1}=d
\quad\text{and}\quad
\sum\limits_{i=2,3}(d_{k_i}+1)=d+1,
\]
where $k_i\neq k_j$ for $i\neq j$ and $k_i\in\{1,2,3\}$. Further, permutations of functions of the form $f=(c\alpha,\alpha,\beta)$, where $c\in\mathbb{C}$ is a unimodular constant and $z\mapsto(\alpha(z),\beta(z))$ is a holomorphic isometry from $(\mathbb{D},ds^2_{\mathbb{D}})$ to $(\mathbb{D},ds^2_{\mbb D})^2$, are all the examples of isometries in this case.
\end{proof}

\begin{proof}[Case 7.] Exactly one of the sets $A_i$ is empty, and the other two have non-empty intersection.

\medskip

Without loss of generality, let $A_1\cap A_2\neq \emptyset$ and $A_3=\emptyset$.
That is, there is a point $b\in\partial\mathbb{D}$ which is not a branch point of $f$ such that $\vert f_i(z)\vert\rightarrow 1$ for $i=1,2$ and $\vert f_3(z)\vert\rightarrow r<1$ as $z\rightarrow b$. 
Since $f(0)=0$, we have
\[
\left(1-f_1(z)\overline{f_1(b)}\right)^{d_1+1}\left(1-f_2(z)\overline{f_2(b)}\right)^{d_2+1}\left(1-f_3(z)\overline{f_3(b)}\right)^{d_3+1}=\left(1-z\bar{b}\right)^{d+1}.
\]
The vanishing order of $1-f_i(z)\overline{f_i(b)}$ at $z=b$ equals to one for both $i=1,2$. Therefore, on comparing the vanishing order at $b$ in the equation, we obtain 
\[
\sum_{i=1,2}(d_i+1)=d+1. 
\]
Observe that if $A_1\setminus A_2\neq\emptyset$, then $d_1=d$, and hence $d_2=-1$, which is a contradiction. Therefore, $A_1\setminus A_2=\emptyset$. Similarly, $A_2\setminus A_1=\emptyset$ and hence $A_1=A_2$. 

\medskip

By the hypothesis, $A_3=\emptyset$. Thus, if $b\in\partial \mathbb{D}$ is not a branch point of $f$, then $\vert f_i(z)\vert\rightarrow 1$ as $z\rightarrow b$ for both $i=1,2$. Since $f_1$ and $f_2$ are continuous functions from $\overline{\mathbb{D}}$ to $\overline{\mathbb{D}}$ and $f$ has only finitely many branch points, $\vert f_i(z)\vert\rightarrow 1$ whenever $z\rightarrow b\in\partial\mathbb{D}$ for both $i=1,2$. Thus, $f_1$ and $f_2$ are proper holomorphic maps and hence finite Blaschke products. As noted earlier, neither $f_1$ nor $f_2$ can have poles in $\mathbb{C}$. Therefore, $f_1(z)=c_1z^{m}$ and $f_2(z)=c_2 z^{n}$ for some $\vert c_i\vert=1$ and $m,\,n\in\mathbb{Z}^+$. Thus, the functional equation becomes
\[
\left(\frac{1}{1-\vert f_3(z)\vert^2}\right)^{d_3+1}=\left(\frac{1-\vert z\vert^{2m}}{1-\vert z\vert^2}\right)^{d_1+1}\left(\frac{1-\vert z\vert^{2n}}{1-\vert z\vert^2}\right)^{d_2+1}.
\]
The right hand side is a finite sum of the norm squares of monomial functions and the left hand side is an infinite sum of the norm squares of linearly independent holomorphic functions, as $f_3$ is non-constant. Hence, a contradiction. Therefore, this case is not possible.
\end{proof}

\begin{proof}[Case 8.] Exactly two of the sets $A_i$ are empty.

\medskip

Without loss of generality, let $A_1=A_2=\emptyset$.
That is, if $b\in\partial \mathbb{D}$ is not a branch point of $f$, then $\vert f_i(z)\vert\rightarrow r_i(b)<1$ for $i=1,2$ and $\vert f_3(z)\vert\rightarrow 1$ as $z\rightarrow b$. 
Since $f(0)=0$, we have
\[
\left(1-f_1(z)\overline{f_1(b)}\right)^{d_1+1}\left(1-f_2(z)\overline{f_2(b)}\right)^{d_2+1}\left(1-f_3(z)\overline{f_3(b)}\right)^{d_3+1}=\left(1-z\bar{b}\right)^{d+1}.
\]
The vanishing order of $1-f_3(z)\overline{f_3(b)}$ at $z=b$ equals to $1$. Therefore, on comparing the vanishing order at $b$ in the equation,
we have $d_3=d$. Since $f_3$ is a continuous function from $\overline{\mathbb{D}}$ to $\overline{\mathbb{D}}$ and $f$ has only finitely many branch points, $\vert f_3(z)\vert\rightarrow 1$ whenever $z\rightarrow b\in\partial\mathbb{D}$. Thus, $f_3$ is a proper holomorphic map and hence a finite Blaschke product. As observed earlier, $f_3$ cannot have poles in $\mathbb{C}$. Therefore, $f_3(z)=cz^k$ for some $\vert c\vert=1$ and $k\in\mathbb{Z}^+$. Thus, the functional equation becomes
\[
\left(\frac{1}{1-\vert f_1(z)\vert^2}\right)^{d_1+1}\left(\frac{1}{1-\vert f_2(z)\vert^2}\right)^{d_2+1}=\left(\frac{1-\vert z\vert^{2k}}{1-\vert z\vert^2}\right)^{d+1}.
\]
The right hand side is a finite sum of the norm squares of monomial functions and the left hand side is an infinite sum of the norm squares of linearly independent holomorphic functions, as $f_1$ and $f_2$ are non-constant. Hence, a contradiction. Therefore, this case is not possible.
\end{proof}

\section{Proof of Theorem \ref{1.7}}

\noindent Consider the unit ball $\mathbb{B}^n \subset \mathbb{C}^n$. The Bergman kernel of $\mathbb{B}^n$ is 
\[
K_{\mathbb{B}^n}(z,w)=\frac{n!}{\pi^n}\frac{1}{(1-z\cdot\ov{w})^{n+1}},
\]
where $z\cdot\bar{w}=z_1\ov{w}_1+\cdots+z_n\ov{w}_n$. The weight is
\[
\mu_{\mbb B^n,d}(z)= K_{\mbb B^n}^{-d}(z)=\left(\frac{\pi^{n}}{n!}\right)^d (1-\left\vert z\right\rvert^2)^{nd+d}, \quad z\in\mbb B^n
\]
and a computation shows that the weighted Bergman kernel $K_{\mathbb{B}^n,d}$ for a non-negative integer $d$ is 
\[
K_{\mbb B^n,d}(z,w)=\binom{nd+d+n}{n}\,\frac{(n!)^{d+1}}{\pi^{n(d+1)}}\,\frac{1}{(1-z\cdot \bar{w})^{(n+1)(d+1)}}\quad \quad\text{for } z,w\in\mathbb{B}^n.
\]

\medskip

Let $f : (\mbb D, ds^2_d) \ra (\mbb D, ds^2_{d_1}) \times (\mbb B^n, ds^2_{d_2})$ be a holomorphic isometry, where $d, d_1, d_2$ are non-negative integers and $ds^2_{d_2}=ds^2_{\mbb B^n,d_2}$. Reparametrize $f$ so that $f(0) = 0$. We will write $K_{d}$, $K_{d_1}$, $K_{d_2}$ instead of $K_{\mbb D,d}$, $K_{\mbb D,d_1}$ and $K_{\mbb B^n,d_2}$ respectively, for brevity. Let $f=(g,h)$ and $h=(h_1,\ldots, h_n)$. Then, $ds^2_{d_1}(g(z)) + ds^2_{d_2}(h(z))=ds^2_{d}(z)$. That is,
\begin{multline*}
\left(\frac{\partial^2}{\partial z\partial\bar{z}}\log K_{d_1}\right)(g(z))\, \vert g'(z)\vert^2 \,dz \otimes d\ov{z}
+
\sum_{i,j=1}^n \left(\frac{\partial^2}{\partial z_i\partial\bar{z}_j}\log K_{d_2}\right)(h(z)) \,h_i'(z)\ov{h_j'(z)} \,dz\otimes d\ov{z}\\
=
\frac{\partial^2}{\partial z\partial\bar{z}}\log K_{d}(z) \,dz \otimes d\ov{z}.
\end{multline*}
which is equivalent to 
\begin{equation}
\frac{\partial^2}{\partial z\partial\bar{z}}\log K_{d_1}(g(z))+\frac{\partial^2}{\partial z\partial\bar{z}}\log K_{d_2}(h(z))=\frac{\partial^2}{\partial z\partial\bar{z}}\log K_{d}(z), \quad z\in \mathbb{D}.
\end{equation}
By combining similar arguments in the proof of the first part of Theorem \ref{1.5} with the explicit expression for $K_{\mbb B^n,d}$, it follows that
\begin{equation}
\left(1-\left\lvert g(z)\right\rvert^2\right)^{2d_1+2}\left(1-\sum\limits_{i=1}^n\left\lvert h_i(z)\right\rvert^2\right)^{(n+1)(d_2+1)}=\left(1-\left\lvert z\right\rvert^2\right)^{2d+2}.
\end{equation}

On expanding as a power series and comparing the coefficients of $\left\lvert z\right\rvert^2$ on both sides of the above equation, we get
\begin{equation}
2(d_1+1)\left\lvert g'(0)\right\rvert^2 + (n+1)(d_2+1)\left(\sum_{i=1}^n\left\lvert h_i'(0)\right\rvert^2\right)=2(d+1).
\end{equation}
Since $g:\mathbb{D}\rightarrow\mathbb{D}$ and $h:\mbb D\ra \mbb B^n$ are  holomorphic maps, it follows by the Schwarz lemma that $\vert g'(0)\vert\leq 1$ and $\sum_{i=1}^n\left\lvert h_i'(0)\right\rvert^2 \leq 1$. Thus,
\[
2(d+1)\leq 2(d_1+1)+(n+1)(d_2+1).
\]
Here, equality holds if and only if $\vert g'(0)\vert=1$ and
$\sum_{i=1}^n\left\lvert h_i'(0)\right\rvert^2=1$. By the Schwarz lemma, $\vert g'(0)\vert=1$ only if $g(z)=cz$ for some $\vert c \vert=1$.  Saying $\sum_{i=1}^n\left\lvert h_i'(0)\right\rvert^2=1$ is equivalent to saying that $h'(0)$ is an isometry of $\mathbb{C}$ into $\mathbb{C}^n$, in which case, the Schwarz lemma implies that $h(z)=h'(0)z=(h_1'(0) z,\ldots,h_n'(0)z)$ for all $z\in \mathbb{D}$. 
Hence, equality holds precisely when there exists a unimodular constant $c$ and a vector $(c_1,\ldots,c_n)\in\partial\mbb B^n$ such that $f(z)=(cz,c_1 z,\ldots,c_n z)$, after a possible reparametrization.

\section{Proof of Theorem \ref{1.8}}
Let $f=(f_1,\ldots,f_n):(\mathbb{D},ds^2_{c_1})\times\cdots\times(\mathbb{D},ds^2_{c_m})\rightarrow (\mathbb{D},ds^2_{d_1})\times\cdots\times(\mathbb{D},ds^2_{d_n})$ be a holomorphic isometry. Reparametrize $f$ so that $f(0)=0$. Set
\[
\mu_1(z_1, z_2, \ldots, z_m) =\mu_{c_1}(z_1) \cdot \mu_{c_2}(z_2) \cdots \mu_{c_m}(z_m)
\]
and
\[
\mu_2(z_1, z_2, \ldots, z_n) =\mu_{d_1}(z_1) \cdot \mu_{d_2}(z_2) \cdots \mu_{d_n}(z_n).
\]
Then, as before, $ds^2_{\mbb D^m,\mu_1}(z)=ds^2_{\mbb D^n,\mu_2}(f(z))$ and hence
\begin{equation}\label{pol iso}
    \partial\ov{\partial} \log K_{\mbb D^n,\mu_2}(f(z))=\partial\ov{\partial} \log K_{\mbb D^m,\mu_1}(z),\quad z\in\mbb D^m.
\end{equation}
By combining similar arguments in the proof of Theorem \ref{1.5} with the explicit expression for $K_{\mbb D, d}$, $d \ge 0$, it follows that
\begin{equation}\label{fxn pol}
\prod_{i=1}^n \left(1-\left\lvert f_i(z)\right\rvert^2\right)^{d_i+1}=\prod_{j=1}^m \left(1-\left\lvert z_j\right\rvert^2\right)^{c_j+1}.
\end{equation}

In particular, (\ref{fxn pol}) implies that $f$ must be a proper map. 

\medskip

Fix $(z_2,\ldots,z_m)\in\mathbb{D}^{m-1}$. From (\ref{fxn pol}), we also see that $z_1\mapsto f(z_1,\ldots,z_m)$ is a holomorphic isometry of $(\mathbb{D},ds^2_{c_1})$ into $(\mathbb{D},ds^2_{d_1})\times\cdots\times(\mathbb{D},ds^2_{d_n})$ for every $(z_2,\ldots,z_m)\in\mathbb{D}^{m-1}$. From 
\[
\prod_{i=1}^n \left(1-\left\lvert f_i(z_1,\ldots,z_m)\right\rvert^2\right)^{d_i+1}=\prod_{j=1}^m \left(1-\left\lvert z_j\right\rvert^2\right)^{c_j+1},
\]
we also see that this isometry does not preserve the origin, except when $z_2=\cdots=z_m=0$.

\begin{lem}\label{not iden zero}
Fix $i \le n$. If $\frac{\partial f_i}{\partial z_1} \not \equiv 0$, then
\[
f_i(0,z_2,\ldots,z_m)=0
\]
for all $(z_2,\ldots,z_m)\in\mathbb{D}^{m-1}$.
\end{lem}

\begin{proof}

Fix $(z_2,\ldots,z_m)\in\mathbb{D}^{m-1}$. The map $F^1:z_1\mapsto f(z_1,\ldots,z_m)$ is an isometric embedding with component functions $f_l$ satisfying
\[
\prod_{l=1}^n \left(1-\left\lvert f_l(z_1,\ldots,z_m)\right\rvert^2\right)^{d_l+1}=a\left(1-\left\lvert z_1\right\rvert^2\right)^{c_1+1}
\]
for some constant $a>0$. We can argue as in the second half of the proof of Theorem \ref{1.5}. Let $V^1$ denote the irreducible one-dimensional projective analytic subvariety in $\mathbb{P}^1\times(\mathbb{P}^1)^n$ extending the graph of $F^1$ and $\pi$ be the projection of $V^1$ onto the first coordinate. There exists a discrete set $A\subset\mathbb{P}^1$ such that $\pi$, when restricted to $V^1\setminus\pi^{-1}(A)$, is a holomorphic covering with $s$ branches.

\medskip

Suppose $(z_1,a_1,\ldots,a_n)$ is a point on $V^1$ with $z_1\in\mathbb{C}$ and $a_i=\infty$. Let $U\subset\mathbb{C}\setminus \left(A\cup\{z_1\}\right)$ be an open set such that $\pi^{-1}(U)=\bigcup_{l=1}^s W_l$, where $W_l$'s are disjoint. Take any of the $W_l$'s, say $W_1$. For all $(w,f_1,\ldots,f_n)\in W_1$, we have
\[
\prod_{l=1}^n\left(1-a_l\bar{f_l}(w,z_2,\ldots,z_m)\right)^{d_l+1}=a(1-z_1\bar{w})^{c_1+1}.
\]
Since the right-hand side is always finite and $a_i=\infty$, either $f_i\equiv 0$ or $(1-a_j\overline{f_j})\equiv 0$ on $U$ for some $j$. Since $\partial f_i/\partial z_1$ is not identically zero, the former cannot be true. The latter possibility is also ruled out as the right side is not identically zero. This is a contradiction.
Therefore, $f_i(z_1,\ldots,z_m)$, as a function of $z_1$, can take the value infinity only at $z_1=\infty$. Since $f_i$, as a function of $z_1$, is not identically constant, the projection of the analytic subvariety corresponding to the component function $f_i$ onto the second coordinate is all of $\mathbb{P}^1$. Therefore, $f_i(\infty,z_2,\ldots,z_m)=\infty$ for all $(z_2,\ldots,z_m)\in\mbb D^{m-1}$. Since (\ref{V_i}) is invariant under the involution $(z, h) \mapsto (1/\ov z, 1/ \ov h)$, it follows that $f_i(0,z_2,\ldots,z_m)=0$ for all $(z_2,\ldots,z_m)\in\mbb D^{m-1}$. 
\end{proof}

\begin{lem}\label{family of varieties}
Let $W\subset\mathbb{C}$ be an open set and $(f_{w,1}(z),f_{w,2}(z),\ldots,f_{w,n}(z))$ a family of holomorphic functions on $\mathbb{D}$ depending holomorphically on $w\in W$ such that
\[
\prod_{i=1}^n \left( 1-\left\lvert f_{w,i}\right\rvert^2\right)^{a_i+1} = \left(1-\left\lvert z\right\rvert^2\right)^{k+1}
\]
where $a_i$'s and $k$ are non-negative integers. Then, for all $i$, $\frac{\partial f_{w,i}}{\partial w}\equiv 0$.
\end{lem}

\begin{proof}
For every $w\in W$, let $F_w$ denote the isometry with component functions $f_{w,i}$, $1\leq i\leq n$. We can determine the analytic subvariety extending the graph of $F_w$, in the same way as in Proposition \ref{extension of graph}. The proof of this proposition shows that the analytic subvariety is independent of $w$. Therefore, $F_w$ is also independent of $w$. Hence, $\frac{\partial f_{w,i}}{\partial w}\equiv 0$ for all $i$.
\end{proof}

By Lemma \ref{not iden zero}, if for any $i\leq n$, $\frac{\partial f_i}{\partial z_1}$ is not constantly zero, then $f_i(0,z_2,\ldots,z_m)=0$ for all $(z_2,\ldots,z_m)\in\mathbb{D}^{m-1}$. It follows from the functional equation
\begin{equation}\label{pol fxn eqn 1}
\prod_{i=1}^n \left(1-\left\lvert f_i(z_1,\ldots,z_m)\right\rvert^2\right)^{d_i+1}=\prod_{j=1}^m \left(1-\left\lvert z_j\right\rvert^2\right)^{c_j+1},
\end{equation}
that we can not have $f_i(0,z_2,\ldots,z_m)=0$ for all $i\le n$ and $(z_2,\ldots,z_m)\in\mathbb{D}^{m-1}$. After rearranging the component indices, if necessary, we can assume that there exists a postive integer $n_1$ such that 
\[
\frac{\partial f_i}{\partial z_1}\equiv 0 \quad\text{ for all }\quad n_1+1\leq i\leq n.
\]
Then, we can write $f_i(z_1,z_2,\ldots,z_m)=f_i(z_2,\ldots,z_m)$ for all $n_1+1\leq i\leq n$. On substituting $z_1=0$ in (\ref{pol fxn eqn 1}), we get
\begin{equation}\label{pol fxn eqn 2}
\prod_{i=n_1+1}^n \left(1-\left\lvert f_i(z_2,\ldots,z_m)\right\rvert^2\right)^{d_i+1}=\prod_{j=2}^m \left(1-\left\lvert z_j\right\rvert^2\right)^{c_j+1}.
\end{equation}
Dividing (\ref{pol fxn eqn 1}) by (\ref{pol fxn eqn 2}), we obtain
\[
\prod_{i=1}^{n_1} \left(1-\left\lvert f_i(z_1,\ldots,z_m)\right\rvert^2\right)^{d_i+1}=\left(1-\left\lvert z_1\right\rvert^2\right)^{c_1+1}.
\]
So, we now have a holomorphic family (with parameters $z_2,\ldots,z_m$) of holomorphic functions satisfying the above functional equation in $z_1$. It follows from Lemma \ref{family of varieties}, that for all $1\leq i\leq n_1$, $f_i$ is a function of $z_1$ only. Note that $(f_1,\ldots,f_{n_1}):(\mbb D,ds^2_{c_1})\rightarrow (\mbb D, ds^2_{d_1})\times\cdots\times(\mbb D, ds^2_{d_{n_1}})$ is a holomorphic isometry.
We can now argue in the same way for the variable $z_2$ in (\ref{pol fxn eqn 2}). By induction, this completes the proof of the first part of the theorem.

\medskip

It now follows from Theorem \ref{1.5} that for all $1\leq i\leq m$, 
\[
c_i+1\leq \sum_{k=n_1+\cdots+n_{i-1}+1}^{n_1+\cdots+n_{i}}(d_k+1)
\]
and equality holds precisely when $f_k= a_k z_i$ for all $n_1+\cdots+n_{i-1}+1\leq k \leq n_1+\cdots+n_{i}$, where the $a_k$'s are unimodular constants.

\section{Appendix}\label{Appendix}

\noindent Let $d_1, d_2, \ldots, d_n$ be an $n$-tuple of non-negative integers. As before, let $\mu_{d_i}$ denote the weight $K_{\mbb D}^{-d_i}$ on $\mbb D$ and $K_{\mbb D, d_i}$ the weighted Bergman kernel corresponding to the weight $\mu_{d_i}$. Consider the polydisc $\mbb D^n$ endowed with the weight
\[
\mu(z) = \mu(z_1, z_2, \ldots, z_n) = \mu_{d_1}(z_1) \cdot \mu_{d_2}(z_2) \cdots \mu_{d_n}(z_n)
\]
which is the point-wise product of the weights in each factor. Note that each $\mu_{d_i}$ is an admissible weight on $\mbb D$, and the explicit prescription of $\mu$ shows that it is also an admissible weight on $\mbb D^n$.

\medskip

Using the reproducing property, it can be seen that
\[
K_{\mbb D^n, \mu}((z_1, \ldots, z_n), (\z_1, \ldots, \z_n)) = K_{\mbb D, d_1}(z_1, \z_1) \cdots K_{\mbb D, d_n}(z_n, \z_n) 
\]
for $(z_1, \ldots, z_n), (\z_1, \ldots, \z_n) \in \mbb D^n$. In other words, the multiplicative property of the Bergman kernel holds for this class of weights on $\mbb D$. We do not know whether this is true for general admissible weights on arbitrary domains. 

\medskip

What this means for the situation in the proof of second part of Theorem \ref{1.5}  is that the isometry condition  on the germ
\[
g:(\mathbb{D},ds^2_d;x_0)\rightarrow (\mathbb{D},ds^2_{d_1};y_{01})\times\cdots\times(\mathbb{D},ds^2_{d_n};y_{0n})
\]
which can be a priori written as
\[
ds^2_d(z) = ds^2_{d_1}(g_1(z)) + \ldots + ds^2_{d_n}(g_n(z))
\]
or equivalently
\[
\pa \ov \pa \log K_{\mbb D, d}(z) = \pa \ov \pa \log K_{\mbb D, d_1}(g_1(z)) + \ldots + \pa \ov \pa \log K_{\mbb D, d_n}(g_n(z))
\]
for $z \in \mbb D$ close to $x_0$, can now be translated as
\begin{eqnarray*}
\partial\bar{\partial}\log K_{\mathbb{D},d}(z)
&=&\partial\bar{\partial}\left( \sum_{i=1}^n \log K_{\mathbb{D},d_i}(g_i)\right)
= 
\partial\bar{\partial} \log \left( \prod_{i=1}^n K_{\mathbb{D},d_i}(g_i)\right)\\
&=&
\partial\bar{\partial}\log K_{\mathbb{D}^n,\mu}(g),
\end{eqnarray*}
where $\mu(z_1,\ldots,z_n) = \prod_{i=1}^n\mu_{d_i}(z_i)$ is an admissible weight on $\mathbb{D}^n$. Therefore, 
\[
g :(\mathbb{D},ds^2_d;x_0)\rightarrow(\mathbb{D}^n,ds^2_{\mu};y_0)
\]
is a germ of a holomorphic isometry.

\medskip

In order to see that the germ $g$ of an isometric embedding of the unit disc into a polydisc can be extended over the whole unit disc, first apply Kobayashi's embedding of a domain $\Om \subset \mbb C^n$ isometrically in the projectivization of a separable Hilbert space -- the proof from \cite{Kob} works verbatim in the weighted case and we merely state the result that can be used.

\begin{thm}\label{FS}
Let $\Omega\subset\mathbb{C}^n$ be a domain and $\mu$ be an admissible weight on $\Omega$. Let $H=\mathcal{O}_{\mu}(\Omega)$ and $H^*$ be its dual space, i.e., the space of all the linear functionals on $H$. Assume that the weight $\mu$ is such that
\begin{enumerate}
\item[A1.] For every $z_1,z_2\in\Omega$, there exists an $f\in H$ such that $f(z_1)=0$ and $f(z_2)\neq 0$.
\item[A2.] For every $z\in\Omega$ and $X\in\mathbb{C}^n$, there exists an $f\in H$ such that $f(z)=0$ and $\vert f'(z)\cdot X\vert\neq 0$.
\end{enumerate}
Let $P(H^*)$ denote the projectivization of the separable Hilbert space $H^{*}$.
Then there exists a canonical complex analytic isometrical imbedding $\Phi:(\Omega,ds^2_{\Omega,\mu})\rightarrow (P(H^*),ds^2_{FS})$, where $ds^2_{FS}$ is the Fubini-Study metric.
\end{thm}

It is not difficult to see that both conditions listed above are satisfied in the case of the product weight $\mu$ on $\mbb D^n$ by using its explicit expression. 

\medskip

In Theorem \ref{FS}, take $\Omega=\mathbb{D}^n$ and $\mu$ as before. Since composition of two isometries is an isometry, the map $\Phi\circ g:(\mathbb{D},ds^2_d;x_0)\rightarrow(P(H^{*}),ds^2_{FS};\Phi(y_0))$ is a local holomorphic isometry. Since $ds^2_d$ is a real-analytic K\"{a}hler metric, it follows by Calabi's theorem \cite{Cb} that $\Phi\circ g$ admits an extension to a holomorphic isometry $\Psi:\mathbb{D}\rightarrow P(H^*)$. 

\medskip

Analytically continuing $\Phi\circ g$ gives an analytic continuation of $Graph(g)$ (which is locally defined near $x_0$ as a germ to begin with) to a complex-analytic subvariety of $\mathbb{D}\times\mathbb{D}^n$, as any point $(z,a_1,\ldots,a_n)$ in the anaytic continuation of $g$ along a loop satisfies
\[
(1-\vert z\vert^2)^{d+1}=\prod_{i=1}^n(1-\vert a_i\vert^2)^{d_i+1}.
\]
Since $(\mathbb{D}^n,\mu)$ is complete and isometrically embedded in $P(H^*)$ via $\Phi$, $\Psi(\mathbb{D})\subset \Phi(\mathbb{D}^n)$. This gives the existence of a holomorphic isometry $F:(\mathbb{D},ds^2_d)\rightarrow (\mathbb{D},ds^2_{d_1})\times \cdots\times (\mathbb{D},ds^2_{d_n})$ that extends the germ of $f$.

\medskip

Going back to the proof of Theorem $\ref{1.5}$, let $(z_0,a_1,\ldots,a_n)\in V\setminus\pi^{-1}(A)$. Choose a neighborhood $U_0$ of $z_0$ lying in $\mathbb{D}$ or $\mathbb{P}^1\setminus\ov{\mathbb{D}}$ depending upon whether $z_0\in\mathbb{D}$ or $z_0\in\mathbb{P}^1\setminus\ov{\mathbb{D}}$. There exist holomorphic functions $h_i$ on $U_0$ with $(z,h_1(z),\cdots,h_n(z))\in V$ for all $z\in U_0$. Since
\[
\prod_{i=1}^n \left( 1-\vert h_i(z)\vert^2\right)^{d_i+1}=\left(1-\vert z\vert^2\right)^{d+1},
\]
either $h_i(U_0)\subset\mathbb{D}$ or $h_i(U_0)\subset\mathbb{P}^1\setminus \ov{\mathbb{D}}$ for $1\leq i\leq n$. We now study the two cases separately.

\begin{proof}[Case 1.] Let $z_0\in\mathbb{D}$.
Define holomorphic functions $g_i$ on $U_0$ as follows
\[
g_i(z)=
\begin{cases}
h_i(z)&\text{if }\,
h_i(U_0)\subset \mathbb{D}\\
1/h_i(z)&\text{if }\,
h_i(U_0)\subset\mathbb{P}^1\setminus\ov{\mathbb{D}}
\end{cases}.
\]
Assume that only $h_1(U_0)\subset\mathbb{P}^1\setminus\ov{\mathbb{D}}$. Note that $d_1$ must be an odd. So, we obtain
\[
\sum_{i=1}^n (d_i+1)\log(1-\vert g_i(z)\vert^2)-(d_1+1)\log \vert g_1(z)\vert^{2}=(d+1)\log (1-\vert z\vert^2).
\]
Applying $\partial\bar{\partial}$ to both sides, the second term on the left hand side vanishes and we get
\[
\sum_{i=1}^n(d_i+1)\frac{\vert g'_i(z)\vert^2}{(1-\vert g_i(z)\vert^2)^2}=\frac{d+1}{(1-\vert z\vert^2)^2},
\]
which is same as saying
\[
\sum_{i=1}^n\partial\bar{\partial}\log K_{\mathbb{D},d_i}(g_i)=\partial\bar{\partial}\log K_{\mathbb{D},d}(z).
\]
Other cases can be handled similarly to see that $(g_1,\ldots,g_n)$ is a germ of a holomorphic isometry and hence can be extended to a holomorphic isometry over the whole unit disc. Thus, we also have the extension of the functions $h_i$ over the whole unit disc such that $(z,h_1(z),\ldots,h_n(z))\in V$ for all $z\in\mathbb{D}$.
\end{proof}

\begin{proof}[Case 2.] Let $z_0\in \mathbb{P}^1\setminus\ov{\mathbb{D}}$.
Let $V_0=\{z\in\mathbb{D}: 1/z\in U_0\}$. Define holomorphic functions $g_i$ on $V_0$ as follows
\[
g_i(z)=
\begin{cases}
h_i(1/z)&\text{if }\,
h_i(U_0)\subset \mathbb{D}\\
1/h_i(1/z)&\text{if }\,
h_i(U_0)\subset\mathbb{P}^1\setminus\ov{\mathbb{D}}
\end{cases}.
\]
Assume that only $h_1(U_0)\subset\mathbb{P}^1\setminus\ov{\mathbb{D}}$. Note that both $d_1$ and $d$ must be either odd or even. 
For all $z\in V_0$, we therefore obtain
\[
\sum_{i=1}^n (d_i+1)\log(1-\vert g_i(z)\vert^2)-(d_1+1)\log \vert g_1(z)\vert^{2}=(d+1)\log (1-\vert z\vert^2)-(d+1)\log \vert z\vert^{2}.
\]
Applying $\partial\bar{\partial}$ to both sides, the second term on both the sides vanish and we get
\[
\sum_{i=1}^n(d_i+1)\frac{\vert g'_i(z)\vert^2}{(1-\vert g_i(z)\vert^2)^2}=\frac{d+1}{(1-\vert z\vert^2)^2},
\]
which is same as saying
\[
\sum_{i=1}^n\partial\bar{\partial}\log K_{\mathbb{D},d_i}(g_i)=\partial\bar{\partial}\log K_{\mathbb{D},d}(z).
\]
Other cases can be handled similarly to see that $(g_1,\ldots,g_n)$ is a germ of a holomorphic isometry and hence can be extended to a holomorphic isometry over the whole unit disc. Thus, we also have the extension of the functions $h_i$ over the whole $\mathbb{P}^1\setminus\ov{\mathbb{D}}$ such that $(z,h_1(z),\ldots,h_n(z))\in V$ for all $z\in\mathbb{P}^1\setminus\ov{\mathbb{D}}$. 

\end{proof}
\section{Concluding Remarks and Questions}

Theorems $\ref{1.5}, \ref{1.6}, \ref{1.7}$ and $\ref{1.8}$ all deal with isometries with respect to the weighted metrics $ds^2_{\mbb D, d}$. This metric provides a canonical class of metrics that give rise to conformal constants in the functional equations that arise from the isometry condition. In the same way, it would be interesting to explore a weighted analog of \cite{Ng1} that deals with isometries from the ball into the product of two balls, all living in possibly different dimensions. Further, Mok--Ng \cite{MokNg} and Yuan \cite{Yuan} have also considered germs of measure preserving holomorphic maps. Obtaining analogs of these results in the context of measures arising from the volume forms associated with $ds^2_{D, d}$ also seems to be interesting. 

\medskip




\begin{thebibliography}{MokNg}

\bibitem{Cb} Calabi, Eugenio:
\emph{Isometric imbedding of complex manifolds}, Ann. of Math. 
{\bf 58} (1953), 1--23.

\bibitem{CXY}  Chan, Shan Tai; Xiao, Ming; Yuan, Yuan:
\emph{Holomorphic isometries between products of complex unit balls}, Internat. J. Math. {\bf 28} (2017), no. 9, 1740010, 22 pp.

\bibitem{CY}  Chan, Shan Tai; Yuan, Yuan:
\emph{Holomorphic isometries from the Poincar\'e disk into bounded symmetric domains of rank at least two}, Ann. Inst. Fourier (Grenoble) {\bf 69} (2019), no. 5, 2205–2240.

\bibitem{Kob} Kobayashi, Shoshichi:
\emph{Geometry of bounded domains}, Trans. Amer. Math. Soc. {\bf 92} (1959), 267–290.

\bibitem{Mok1} Mok, Ngaiming:
\emph{Extension of germs of holomorphic isometries up to normalizing constants with respect to the Bergman metric}, J. Eur. Math. Soc. {\bf 14} (2012), no. 5, 1617--1656.

\bibitem{MokNg} Mok, Ngaiming; Ng, Sui Chung:
\emph{Germs of measure-preserving holomorphic maps from bounded symmetric domains to their Cartesian products}, J. Reine Angew. Math. {\bf 669} (2012), 47–-73. 

\bibitem{Ng1} Ng, Sui-Chung:
\emph{On holomorphic isometric embeddings of the unit $n$-ball into products of two unit $m$-balls}. Math. Z. {\bf 268} (2011), no. 1-2, 347–354. 

\bibitem{Ng} Ng, Sui-Chung:
\emph{On holomorphic isometric embeddings of the unit disk into polydisks}, Proc. Amer. Math. Soc. {\bf 138} (2010), no. 8, 2907–-2922.

\bibitem{PW2} Pasternak-Winiarski, Zbigniew:
\emph{On weights which admit the reproducing kernel of Bergman type}, Internat. J. Math. Math. Sci. {\bf 15} (1992), no. 1, 1–-14.

\bibitem{PW1} Pasternak-Winiarski, Zbigniew:
\emph{On the dependence of the reproducing kernel on the weight of integration}, J. Funct. Anal. {\bf 94} (1990), no. 1, 110--134.

\bibitem{Yuan}  Yuan, Yuan:
\emph{On local holomorphic maps preserving invariant $(p,p)$-forms between bounded symmetric domains}. Math. Res. Lett. {\bf 24} (2017), no. 6, 1875–1895
\end{thebibliography}
\end{document}